\documentclass[11pt,twoside]{article}

\usepackage{mathrsfs,amsfonts,amsmath,amssymb}
\usepackage{latexsym}
\newtheorem{satz}{Theorem}

\newtheorem{theorem}[satz]{Theorem}
\newtheorem{lemma}[satz]{Lemma}

\newtheorem{corollary}[satz]{Corollary}
\newtheorem{remark}[satz]{Remark}

\def\eps{\varepsilon}
\def\_phi{\varphi}

\def\a{\alpha}
\def\d{\delta}

\def\F{{\mathbb F}}

\def\t{\tilde}
\def\o{\omega}
\def\ov{\overline}

\def\C{{\mathbb C}}
\def\R{{\mathbb R}}
\def\E{\mathsf {E}}
\def\T{{\mathbb T}}

\def\Z_N{{\mathbb Z}_N}
\def\Z{{\mathbb Z}}

\def\f{{\mathbb F}}
\def\Gr{{\mathbf G}}

\def\D{{\mathbb D}}

\def\oM{{\rm M}}

\def\G{\Gamma}
\def\FF{\widehat}
\def\c{\circ}
\def\D{\Delta}

\def\T{\mathsf {T}}

\author{Shkredov I.D.}
\title{ Some remarks on the asymmetric sum--product phenomenon 
	\footnote{
		This work is supported by the Russian Science Foundation under grant 14-11-00433.}
}
\date{}
\begin{document}
	\maketitle

\begin{center}
	Annotation.
\end{center}

{\it \small
	Using
	some new observations  connected to higher energies, 
	we obtain 
	quantitative lower bounds on 
	$\max\{|AB|, |A+C| \}$ and $\max\{|(A+\a)B|, |A+C|\}$, $\a\neq 0$  
	in the regime when the sizes of finite subsets $A,B,C$ 
	of a field 
	differ significantly. 
}
\\

	\section{Introduction}
	\label{sec:introduction}

	Let $p$ be a prime number and $A,B\subset \F_p = \Z/p\Z$ be finite sets.
	Define the  \textit{sum set}, the \textit{difference set},
	the \textit{product set} and the \textit{quotient set} of $A$ and $B$ as
	$$A+B:=\{a+b ~:~ a\in{A},\,b\in{B}\}\,, \quad \quad \quad  A-B:=\{a-b ~:~ a\in{A},\,b\in{B}\}$$
	$$AB:=\{ab ~:~ a\in{A},\,b\in{B}\}\,,\quad  \quad \quad A/B:=\{a/b ~:~ a\in{A},\,b\in{B},\,b\neq0\}\,,$$
	correspondingly.
	One of the central problems in arithmetic combinatorics \cite{TV} it is the
	{\it sum--product problem}, which asks  for estimates of the form 
	\begin{equation}\label{f:sum-product}
	\max \{ |A+A|, |AA| \} \ge |A|^{1+c} 
	\end{equation}
	for some positive $c$. 
	This question was originally posed by Erd\H{o}s and Szemer\'edi \cite{ES} for finite sets of integers; they conjectured that (\ref{f:sum-product}) holds for all $c< 1$. The sum--product problem has since been studied over a variety of fields and rings, see, e.g. 
	\cite{Bourgain_01}, \cite{Bourgain_Z_q}, \cite{Bourgain_ideal}, \cite{BC}, \cite{BKT}, \cite{ES}, \cite{TV} and others.
	We focus on the case of $\F_p$ (and sometimes consider $\R$), where the first estimate of the form (\ref{f:sum-product}) was proved
	by Bourgain, Katz, and Tao \cite{BKT}. 
	At the moment the best results in this direction  are contained in \cite{RRS} and in \cite{KS2}.

	In this 
	article 
	we study an asymmetric variant of the sum--product question
	("the sum--product theorem in $\F_p$ for sets of distinct sizes")
	in the spirit of fundamental paper \cite{Bourgain_more}. 
	Let us 
	recall two results from here.

	\begin{theorem}
		Given $0<\eps<1/10$, there is $\d>0$ such that the following holds.
		Let $A\subset \F_p$ and 
		$$
		p^\eps < |A| < p^{1-\eps} \,.
		$$
		Then either 
		$$
		|AB| > p^\d |A| \quad \mbox{ for all } \quad  B\subset \F_p\,, |B| > p^\eps 
		$$
		or 
		$$
		|A+C| > p^\d |A| \quad \mbox{ for all } \quad  C\subset \F_p\,, |C| > p^\eps \,.
		$$
		\label{t:Bourgain_ABC}
	\end{theorem}

	\begin{theorem}
		Given $0<\eps<1/10$, there is $\d>0$ such that the following holds.
		Let $A\subset \F_p$ and 
		$$
		p^\eps < |A| < p^{1-\eps} \,.
		$$
		Then for any $x\neq 0$ either 
		$$
		|AB| > p^\d |A| \quad \mbox{ for all } \quad  B\subset \F_p\,, |B| > p^\eps
		$$
		or 
		$$
		|(A+x)C| > p^\d |A| \quad \mbox{ for all } \quad  C\subset \F_p\,, |C| > p^\eps \,.
		$$
		\label{t:Bourgain_ABC_shift}
	\end{theorem}

	Theorems \ref{t:Bourgain_ABC}, \ref{t:Bourgain_ABC_shift} were derived in \cite{Bourgain_more} from the following result of paper \cite{Bourgain_DH}. 
	Given a set $A\subseteq \F_p$ denote by $\T^{+}_k (A) := |\{ (a_1,\dots,a_k,a'_1,\dots,a'_k) \in A^{2k} ~:~ a_1 + \dots + a_k = a'_1 + \dots + a'_k \}|$.
	We write $\E^{+} (A)$ for $\T^{+}_2 (A)$.

	\begin{theorem}
		For a positive integer $Q$, there are a positive integer $k$ and a real $\tau >0$ such that if $H\subseteq \F_p^*$ and 
		$$
		|HH| < |H|^{1+\tau}
		$$
		then
		$$
		\T^{+}_k (H) < |H|^{2k} (p^{-1+1/Q} + c_Q |H|^{-Q}) \,,
		$$ 	
		where $c_Q > 0$ depends on $Q$ only. 
		\label{t:Bourgain_T_k}
	\end{theorem}

	The aim of this paper is to obtain 
	explicit 
	bounds in the theorems above. 
	Our arguments are different
	and  
	more elementary than in \cite{Bourgain_more}, \cite{BGK} and \cite{Garaev_survey}. 
	In the proof we almost do not use the Fourier approach and hence the container group $\F_p$.
	That is why we do not need in lower bounds for sizes of $A,B,C$ in terms of the characteristic $p$ but, of course, these sets must be comparable somehow. 
	Also, the arguments work in $\R$ as well and it differs this article  from paper \cite{Bourgain_more}, say.
	Let us formulate our variant of Theorems \ref{t:Bourgain_ABC}, \ref{t:Bourgain_ABC_shift} (see Corollary \ref{c:ABC_my} below).
	One can show 
	that Theorem  \ref{t:ABC_my_intr} implies both of these results if $|A| < p^{1/2-\eps}$, say, 
	see Remark \ref{r:implication} from section \ref{sec:small_prod}.

	\begin{theorem}
		Let $A,B,C\subseteq \F_p$ be arbitrary sets, 
		and $k\ge 1$ be such that 
		$|A| |B|^{1+\frac{(k+1)}{2(k+4)}2^{-k}} \le p$ 
		and 
		\begin{equation}\label{cond:ABC_my_intr}
		|B|^{\frac{k}{8}+ \frac{1}{2(k+4)}} \ge |A| \cdot  C^{(k+4)/4}_* \log^{k}  (|A||B|) \,.
		\end{equation}
		where $C_* >0$ is an absolute constant. 
		Then
		\begin{equation}\label{f:ABC_my1_intr}
		\max\{ |AB|, |A+C| \} \ge 2^{-3} |A| \cdot \min\{|C|, |B|^{\frac{1}{2(k+4)} 2^{-k}} \} \,,
		\end{equation}
		and for any $\a \neq 0$
		\begin{equation}\label{f:ABC_my2_intr}
		\max\{ |AB|, |(A+\a)C| \} \ge 2^{-3} |A| \cdot \min\{|C|, |B|^{\frac{1}{2(k+4)} 2^{-k}} \} \,.
		\end{equation}
		\label{t:ABC_my_intr}
	\end{theorem}

	Actually, we prove that the lower bounds for 
	$|A+C|$, $|(A+\a)C|$ in (\ref{f:ABC_my1_intr}), (\ref{f:ABC_my2_intr})
	could be replaced by similar upper bounds for the energies $\E^{+} (A,C)$, $\E^{\times} (A+\a,C)$, 
	see the second part of Corollary \ref{c:ABC_my} from section \ref{sec:small_prod}. 
	We call Theorem \ref{t:ABC_my_intr}  an  asymmetric sum--product result  because $A$ can be much larger than $B$ and $C$ 
	(say, $|A| > (|B|C|)^{100}$) in contrast with
	the usual  
	quadratic 
	restrictions 
	which 
	follow from 
	the classical Szemer\'{e}di--Trotter Theorem, see \cite{sz-t}, \cite{TV} for the real setting 
	and see \cite{BKT}, \cite{Garaev_survey}, \cite{Rudnev_pp} for the prime fields. 
	On the other hand, the roles of $B$, $C$ are not symmetric as well. 
	The thing is that the method of the proof intensively uses the fact that if $|AB|$ is small comparable to $|A|$, then, roughly speaking, 
	for any integer $k$ 
	size of $(kA)B$ is small comparable to $kA$, roughly speaking (rigorous formulation can be found in section \ref{sec:small_prod}).
	Of course this observation is not true more in any sense  if we replace $\times$ to $+$ and vice versa.

	Also, we obtain a "quantitative"\, version of Theorem \ref{t:Bourgain_T_k}.

	\begin{theorem}
		Let $A,B \subseteq \F_p$  be sets, $M\ge 1$ be a real number and 
		$|AB| \le M|A|$. 
		Then for any $k\ge 2$, 
		$2^{16 k} M^{2^{k+1}} C^2_* \log^8 |A| \le |B|^{}$,  
		one has 
		\begin{equation}\label{f:T_k_G_M_intr}
		\T^{+}_{2^k} (A) 
		\le 2^{4k+6} C_* \log^4 |A| \cdot \frac{M^{2^{k}}|A|^{2^{k+1}}}{p} +  16^{k^2} M^{2^{k+1}} C_*^{k-1}  \log^{4(k-1)} |A| \cdot 
		|A|^{2^{k+1} - 4} |B|^{-\frac{(k-1)}{2}} \E^{+} (A) \,.
		\end{equation}
		Here $C_*>0$ is an absolute constant.
		\label{t:T_k_G_M_intr}
	\end{theorem}

	As a by-product we obtain the best constants in the problem of estimating of the exponential sums over multiplicative subgroups \cite{Bourgain_DH}, \cite{Garaev_survey} and 
	relatively good bounds in the question on basis properties of  multiplicative subgroups \cite{GK}.
	Also, we find a wide series of "superquadratic expanders in $\R$" \cite{BR-NZ} with four variables, see Corollary \ref{c:superquadratic}.

	In contrast to paper \cite{Bourgain_more} we  prove 
	Theorem \ref{t:ABC_my_intr} and Theorem \ref{t:T_k_G_M_intr} 
	independently.
	We realise that Theorem \ref{t:ABC_my_intr} is equivalent to estimating another sort of energies, namely, 
	$$\E^{+}_k (A):= |\{ (a_1,\dots,a_k,a'_1, \dots, a'_k) \in A^{2k} ~:~ a_1-a'_1 = \dots = a_k - a'_k \}|$$ 
	(see the definitions in section \ref{sec:definitions}).
	Thus, a new feature of this paper is an upper bound for $\E^{+}_k (A)$ for sets $A$ with $|AB|\ll |A|$ for some large $B$, 
	see Theorem \ref{t:QG} below. 
	Such upper bound can be of independent interest. 
	
	\begin{theorem}
		Let $A,B \subseteq \F_p$ be two sets, $k\ge 0$ be an integer, and put $M:=|AB^{k+1}|/|A|$.  
		Then for any  $k\ge 0$  such that 
		$$|B|^{k/8+1/2} \ge |A| \cdot  M^{2^k+1} 2^{3k+1}  C^{(k+4)/4}_* \log^{k}  |AB^k| \,,$$ 
		where $C_* >0$ is an absolute constant, 
		we have 
		\begin{equation}\label{f:QG'_intr}
		\E^{+}_{2^{k+1}} (A) \le 2|AB^k|^{2^{k+1}} 
		\,.
		\end{equation}
		\label{t:QG_intr}
	\end{theorem}

	Our approach develops the ideas from \cite{Bourgain_more}, \cite{Shkr_H+L} (see especially section 4 from here) and uses several  sum--product observations of course. 
	We avoid to repeat  combinatorial arguments of Bourgain's paper \cite{Bourgain_more} (although   we use a similar inductive  strategy of the proof) 
	but the method 
	relies on recent geometrical sum--product bounds from Rudnev's article \cite{Rudnev_pp} and further papers as \cite{AMRS}, \cite{MPR-NRS}, \cite{RRS}, \cite{s_E_k} and others.
	In some sense we introduce a new 
	approach 
	of estimating moments $\oM_{k} (f)$ (e.g., $\T^{+}_k (H)$ in Theorem \ref{t:Bourgain_T_k} or $\E^{+}_k (A)$ in Theorem \ref{t:QG_intr}) 
	of some specific functions $f$ : instead of calculating $\oM_{k} (f)$ in terms of 
	suitable 
	norms of $f$, we comparing  $\oM_{k} (f)$ and $\oM_{k/2} (f)$.
	If  $\oM_{k} (f)$ is much less than $\oM_{k/2} (f)$, then  we use induction  and
	if not then thanks 
	some special nature of the function $f$ we deriving from this fact  that
	the additive energy $\E^{+}$ 
	of a level set of $f$ is huge and it gives a contradiction.  
	Clearly, this process can be applied at most $O(\log k)$ number of times and that is why we usually have logarithmic savings
	(compare the index in $\T^{+}_{2^k} (A)$ and the gain $|B|^{-\frac{(k-1)}{2}}$ in estimate (\ref{f:T_k_G_M_intr}), say).

	The paper is organized as follows.
	Section \ref{sec:definitions} contains all required definitions. 
	In section \ref{sec:preliminaries} we give a list of the results, which will be further used  in the text.
	In the next section we consider a particular case of multiplicative subgroups $\G$ and obtain an upper 
	estimate 
	for $\T^{+}_k (\G)$.
	It allows us to obtain new upper bounds for the exponential sums over subgroups which are the best at the moment. 
	This technique is developed in section \ref{sec:small_prod} although we 
	avoid to use the Fourier approach as was done in \cite{Bourgain_more} and in the previous section \ref{sec:proof}.  
	The last section \ref{sec:small_prod} contains all main Theorems \ref{t:ABC_my_intr}--\ref{t:QG_intr}.


	
	The author thanks Misha Rudnev and Sophie Stevens for careful reading of the first draft of this paper and for useful discussions.

\section{Notation}
\label{sec:definitions}


In this paper $p$ is an odd prime number, 
$\F_p = \Z/p\Z$ and $\F_p^* = \F_p \setminus \{0\}$. 
%
%
We denote the Fourier transform of a function  $f : \F_p \to \C$ by~$\FF{f},$
\begin{equation}\label{F:Fourier}
\FF{f}(\xi) =  \sum_{x \in \F_p} f(x) e( -\xi \cdot x) \,,
\end{equation}
where $e(x) = e^{2\pi i x/p}$. 
We rely on the following basic identities
\begin{equation}\label{F_Par}
\sum_{x\in \F_p} |f(x)|^2
=
\frac{1}{p} \sum_{\xi \in \F_p} \big|\widehat{f} (\xi)\big|^2 \,,
\end{equation}
\begin{equation}\label{svertka}
\sum_{y\in \F_p} \Big|\sum_{x\in \F_p} f(x) g(y-x) \Big|^2
= \frac{1}{p} \sum_{\xi \in \F_p} \big|\widehat{f} (\xi)\big|^2 \big|\widehat{g} (\xi)\big|^2 \,,
\end{equation}
and
\begin{equation}\label{f:inverse}
f(x) = \frac{1}{p} \sum_{\xi \in \F_p} \FF{f}(\xi) e(\xi \cdot x) \,.
\end{equation}
Let $f,g : \f_p \to \C$ be two functions.
Put
\begin{equation}\label{f:convolutions}
(f*g) (x) := \sum_{y\in \f_p} f(y) g(x-y) \quad \mbox{ and } \quad
(f\circ g) (x) := \sum_{y\in \f_p} \ov{f(y)} g(y+x)  \,.
\end{equation}
Then
\begin{equation}\label{f:F_svertka}
\FF{f*g} = \FF{f} \FF{g} \quad \mbox{ and } \quad 
\FF{f \circ g} 
= \ov{\FF{f}} \FF{g} \,.
\end{equation}
Put
$\E^{+}(A,B)$ for the {\it common additive energy} of two sets $A,B \subseteq \f_p$
(see, e.g., \cite{TV}), that is, 
$$
\E^{+} (A,B) = |\{ (a_1,a_2,b_1,b_2) \in A\times A \times B \times B ~:~ a_1+b_1 = a_2+b_2 \}| \,.
$$
If $A=B$ we simply write $\E^{+} (A)$ instead of $\E^{+} (A,A)$
and $\E^{+} (A)$ is called the {\it additive energy} in this case. 
Clearly,
\begin{equation*}\label{f:energy_convolution}
\E^{+} (A,B) = \sum_x (A*B) (x)^2 = \sum_x (A \circ B) (x)^2 = \sum_x (A \circ A) (x) (B \circ B) (x)
\end{equation*}
and by (\ref{svertka}),
\begin{equation}\label{f:energy_Fourier}
\E(A,B) = \frac{1}{p} \sum_{\xi} |\FF{A} (\xi)|^2 |\FF{B} (\xi)|^2 \,.
\end{equation}
Also, notice that
\begin{equation}\label{f:E_CS}
\E^{+} (A,B) \le \min \{ |A|^2 |B|, |B|^2 |A|, |A|^{3/2} |B|^{3/2} \} \,.
\end{equation}
In the same way define the {\it common multiplicative energy} of two sets $A,B \subseteq \f_p$
$$
\E^{\times} (A,B) =  |\{ (a_1,a_2,b_1,b_2) \in A\times A \times B \times B ~:~ a_1 b_1 = a_2 b_2 \}| \,. 
$$
Certainly, the multiplicative energy $\E^{\times} (A,B)$ can be expressed in terms of multiplicative convolutions 
similar to (\ref{f:convolutions}).

Sometimes we  use representation function notations like $r_{AB} (x)$ or $r_{A+B} (x)$, which counts the number of ways $x \in \F_p$ can be expressed as a product $ab$ or a sum $a+b$ with $a\in A$, $b\in B$, respectively. 
For example, $|A| = r_{A-A}(0)$ and  $\E^{+} (A) = r_{A+A-A-A}(0)=\sum_x r^2_{A+A} (x) = \sum_x r^2_{A-A} (x)$.  
In this paper we use the same letter to denote a set $A\subseteq \F_p$ and  its characteristic function $A: \F_p \to \{0,1 \}$. 
Thus, $r_{A+B} (x) = (A*B) (x)$, say.

Now consider two families of higher energies. Firstly, let
\begin{equation}\label{def:T_k}
\T^{+}_k (A) := |\{ (a_1,\dots,a_k,a'_1,\dots,a'_k) \in A^{2k} ~:~ a_1 + \dots + a_k = a'_1 + \dots + a'_k \}|
=
\frac{1}{p} \sum_{\xi} |\FF{A} (\xi)|^{2k}
\,.
\end{equation}
Secondly, for $k\ge 2$, we put 
\begin{equation}\label{def:E_k}
\E^{+}_k (A) = \sum_{x\in \F_p} (A\c A)(x)^k = \sum_{x\in \F_p} r_{A-A}^k (x) = 
\E^{+} (\Delta_k (A),A^{k}) \,,
\end{equation}
where 
$$
\Delta_k (A) := \{ (a,a, \dots, a)\in A^k \}\,.
$$ 
Thus, $\E^{+}_2 (A) = \T^{+}_2 (A)= \E^{+} (A)$. 
Also, notice that we always have $\E^{+}_k (A) \ge |A|^k$. 
Finally, let us remark that by definition (\ref{def:E_k}) one has $\E^{+}_1 (A) = |A|^2$.
Some results about the  properties of the energies $\E^{+}_k$ can be found in \cite{SS1}.
Sometimes we use $\T^{+}_k (f)$ and $\E^{+}_k (f)$ for an arbitrary function $f$
and the first  formula from   (\ref{def:E_k}) allows to define $\E^{+}_k (A)$ for any positive $k$.  
It was proved in \cite[Proposition 16]{s_E_k} that  $(\E^{+}_k (f))^{1/2k}$ is a norm for even $k$ and a real function $f$. 
The fact that $(\T^{+}_k (f))^{1/2k}$ is a norm is contained in \cite{TV}.

Let $A$ be a set. 
Put 
$$
R[A] := \left\{ \frac{a_1-a}{a_2-a} ~:~ a,a_1,a_2\in A,\, a_2 \neq a \right\}
$$
and 
$$
Q[A] := \left\{ \frac{a_1-a_2}{a_3-a_4} ~:~ a_1,a_2,a_3,a_4 \in A,\, a_3 \neq a_4 \right\} \,.
$$ 

\bigskip 

All logarithms are to base $2.$ The signs $\ll$ and $\gg$ are the usual Vinogradov symbols.
When the constants in the signs  depend on some parameter $M$, we write $\ll_M$ and $\gg_M$. 
For a positive integer $n,$ we set $[n]=\{1,\ldots,n\}.$ 

\section{Preliminaries}
\label{sec:preliminaries}




We begin with a variation on the famous Pl\"{u}nnecke--Ruzsa inequality, see \cite[Chapter 1]{Ruzsa_book}.  

\begin{lemma}
	Let $\Gr$ be a commutative group. 
	Also, let $A,B_1, \dots, B_h \subseteq \Gr$, $|A+B_j| = \a_j |A|$, $j \in [h]$.
	Then there is a non-empty set  $X\subseteq A$ such that 
	\begin{equation}\label{f:Plunnecke_R}
	|X+B_1+\dots+B_h| \le \a_1 \dots \a_h |X| \,.
	\end{equation}
	Further for any $0< \d < 1$ there is $X \subseteq A$ such that $|X| \ge (1-\d) |A|$ and 
	\begin{equation}\label{f:Plunnecke_XR}
	|X+B_1+\dots+B_h| \le \d^{-h} \a_1 \dots \a_h |X| \,.
	\end{equation}
	\label{l:Plunnecke_R}
\end{lemma}

We need a result from \cite{Rudnev_pp} or see \cite[Theorem 8]{MPR-NRS}.
By the number of point--planes incidences $\mathcal{I} (\mathcal{P}, \Pi)$ between a set of points 
$\mathcal{P} \subseteq \F_p^3$ and a collection of planes $\Pi$  in $\F_p^3$ we mean 
$$
\mathcal{I} (\mathcal{P}, \Pi) := |\{ (p,\pi) \in \mathcal{P} \times \Pi ~:~ p \in \pi \}| \,.
$$

\begin{theorem}
	Let $p$ be an odd prime, $\mathcal{P} \subseteq \F_p^3$ be a set of points and $\Pi$ be a collection of planes in $\F_p^3$. Suppose that $|\mathcal{P}| = |\Pi|$ and that $k$ is the maximum number of collinear points in $\mathcal{P}$. 
	Then the number of point--planes incidences satisfies 
	\begin{equation}\label{f:Misha+}
	\mathcal{I} (\mathcal{P}, \Pi)  \ll \frac{|\mathcal{P}|^2}{p} + |\mathcal{P}|^{3/2} + k |\mathcal{P}| \,.	
	\end{equation}
	\label{t:Misha+}	
\end{theorem}

Notice that in $\R$ we do not need in the first term in estimate (\ref{f:Misha+}).

Let us derive a consequence of Theorem \ref{t:Misha+}.

\begin{lemma}
	Let $A,Q\subseteq \F_p$ be two sets, $A,Q\neq \{0\}$, $M\ge 1$ be a real number, and
	$|QA|\le M|Q|$. 
	Then
	\begin{equation}\label{f:AA_small_energy}
	\E^{+} (Q) \le C_* \left( \frac{M^2 |Q|^4}{p} + \frac{M^{3/2} |Q|^3}{|A|^{1/2}} \right)  \,,
	\end{equation} 
	where $C_* \ge 1$ is an absolute constant. 
	\label{l:AA_small_energy}
\end{lemma}
\begin{proof}
	Put $A=A\setminus \{0\}$.
	We have
	$$
	\E^{+} (Q) = |\{ q_1+q_2=q_3+q_4 ~:~ q_1,q_2,q_3,q_4\in Q\}| 
	\le
	$$
	$$
	\le
	|A_*|^{-2} |\{ q_1+\t{q}_2/a=q_3+\t{q}_4/a' ~:~  q_1,q_3 \in Q,\, \t{q}_2,\t{q}_4 \in QA,\, a,a'\in A_* \}| \,. 
	$$
	The number of the solutions to the last equation can be interpreted as the number of incidences between the set of  points 
	$\mathcal{P} = Q\times QA \times A_*^{-1}$ and planes $\Pi$ with $|\mathcal{P}| = |\Pi| = |A_*||Q||QA|$. Here $k = |QA|$ because $A,Q\neq \{0\}$.
	Using Theorem \ref{t:Misha+} an a trivial inequality $|QA| \le |Q||A|$, we obtain
	$$
	\E^{+} (Q) \ll |A|^{-2} \left( \frac{|A|^2 |Q|^2 |QA|^2}{p} + |Q|^{3/2} |QA|^{3/2} |A|^{3/2} \right)
	\ll
	\frac{M^2 |Q|^4}{p} + \frac{M^{3/2} |Q|^3}{|A|^{1/2}} 
	$$
	as required. 
	$\hfill\Box$
\end{proof}



\bigskip

Finally, we need a combinatorial 

\begin{lemma}
	Let $\Gr$ be a finite abealian group, $A,P$ be subsets of $\Gr$.
	Then for any $k\ge 1$ one has 
	\begin{equation}\label{f:E_P_gen}
	\left( \sum_{x\in P} r^k_{A-A} (x)  \right)^2 \le |A|^{k} \sum_x r^k_{A-A} (x) r_{P-P} (x)  \,.
	\end{equation}
	In particular, 
	\begin{equation}\label{f:E_P_gen+}
	\left( \sum_{x\in P} r^k_{A-A} (x)  \right)^4 \le |A|^{2k} \E^{+}_{2k} (A) \E^{+} (P) \,.
	\end{equation} 
	\label{l:E_P_gen}
\end{lemma}
\begin{proof}
	Clearly, inequality (\ref{f:E_P_gen+}) follows from (\ref{f:E_P_gen}) by the Cauchy--Schwarz inequality.
	To prove estimate (\ref{f:E_P_gen}), we observe that 
	$$
	\left( \sum_{x\in P} r^k_{A-A} (x)  \right)^2
	=  \left( \sum_{x_1,\dots,x_k\in A} |P \cap (A - x_1) \cap \dots \cap (A - x_k)| \right)^2
	\le
	$$
	$$
	\le
	|A|^k \sum_{x_1,\dots,x_k} |P \cap (A - x_1) \cap \dots \cap (A - x_k)|^2
	=
	|A|^k \sum_x r_{P-P} (x) r^k_{A-A} (x)  
	$$
	as required. 
	$\hfill\Box$
\end{proof}

\bigskip 

Combining Theorem \ref{t:Misha+} and Lemma \ref{l:E_P_gen}, we obtain 

\begin{corollary}
	Let $A \subseteq \F_p$, and $B, P\subseteq \F^*_p$ be sets.
	Then for any $k\ge 1$ one has 
	\begin{equation}
	\left( \sum_{x\in P} r^k_{A-A} (x)  \right)^4 \le C_* |A|^{2k} \E^{+}_{2k} (AB) 
	\left( \frac{|P|^4}{p} + \frac{|P|^3}{|B|^{1/2}} \right) \,.
	\end{equation}
	\label{c:change_QG}	
\end{corollary}
\begin{proof}
	By Lemma \ref{l:E_P_gen}, we have
	$$
	\left( \sum_{x\in P} r^k_{A-A} (x)  \right)^2 \le |A|^{k} \sum_x r^k_{A-A} (x) r_{P-P} (x) \,.
	$$
	Further clearly for any $b\in B$ the following holds
	$$
	r_{A-A} (x) \le r_{AB-AB} (xb) \,.
	$$
	Hence
	$$
	\left( \sum_{x\in P} r^k_{A-A} (x)  \right)^2 \le \frac{|A|^{k}}{|B|} \sum_x \sum_{b\in B} r^k_{AB-AB} (xb) r_{P-P} (x)
	=
	\frac{|A|^{k}}{|B|} \sum_x r^k_{AB-AB} (x) r_{B(P-P)} (x) \,.
	$$
	Using the Cauchy--Schwarz inequality, we obtain
	$$
	\left( \sum_{x\in P} r^k_{A-A} (x)  \right)^4
	\le
	\frac{|A|^{2k}}{|B|^2} \E^{+}_{2k} (AB) \sum_x r^2_{B(P-P)} (x) \,.
	$$
	To estimate the sum $\sum_x r^2_{B(P-P)} (x)$ we use Theorem \ref{t:Misha+} as in the proof of Lemma \ref{l:AA_small_energy}. 
	We have
	$$
	\sum_x r^2_{B(P-P)} (x) \le C_* \left( \frac{|B|^2|P|^4}{p} + |B|^{3/2} |P|^3 \right) \,.
	$$
	Thus, 
	$$
	\left( \sum_{x\in P} r^k_{A-A} (x)  \right)^4
	\le
	C_* |A|^{2k} \E^{+}_{2k} (AB) \left( \frac{|P|^4}{p} + \frac{|P|^3}{|B|^{1/2}} \right) \,.
	$$
	This completes the proof. 
	$\hfill\Box$
\end{proof}

\section{Multiplicative subgroups}
\label{sec:proof}


In this section we obtain the best upper bounds for $\T^{+}_k (\G)$, $\E^{+}_k (\G)$ and for the exponential sums over multiplicative subgroups $\G$.  
We begin with the quantity $\T^{+}_k (\G)$. 

\begin{theorem}
	Let $\G \subseteq \F_p^*$  be a multiplicative subgroup.
	Then for any $k\ge 2$, 
	$2^{64 k} C^4_* \le |\G|^{}$ 
	one has 
	\begin{equation}\label{f:T_k_G}
	\T^{+}_{2^k} (\G) \le
	2^{4k+6} C_* \log^4 |\G| \cdot \frac{|\G|^{2^{k+1}}}{p} + 16^{k^2} C^{k-1}_* \log^{4(k-1)} |\G| \cdot 
	|\G|^{2^{k+1} -\frac{(k+7)}{2}} \E^{+} (\G) \,,
	\end{equation}
	where $C_*$ is the absolute constant from Lemma \ref{l:AA_small_energy}. 
	\label{t:T_k_G}
\end{theorem}
\begin{proof}
	Fix any $s \ge 2$. 
	Our intermediate aim is to prove 
	\begin{equation}\label{f:T_2s,T_s}
	\T^{+}_{2s} (\G) \le 32 C_* s^4 \log^4 |\G| \cdot  \left( \frac{|\G|^{4s}}{p} + |\G|^{2s-1/2} \T^{+}_s (\G) \right) 
	\,.
	\end{equation}
	We have 
	$$
	\T^{+}_{2s} (\G) = \sum_{x,y,z} r_{s\G} (x) r_{s\G} (y) r_{s\G} (x+z) r_{s\G} (y+z) 
	\le 
	\frac{8}{5} \sum'_{x,y,z} r_{s\G} (x) r_{s\G} (y) r_{s\G} (x+z) r_{s\G} (y+z) + \mathcal{E} \,,
	$$
	where the sum above is taken over nonzero variables  $x$ with  $r(x) > \T^{+}_{2s} (\G) /(8|\G|^{3s}) := \rho$
	and
	\begin{equation}\label{f:E_bound}
	\mathcal{E} \le 4 r_{s\G} (0) \sum_{y,z} r_{s\G} (y) r_{s\G} (z) r_{s\G} (y+z)
	\le
	4 r_{s\G} (0) |\G|^s \T^{+}_s (\G) 
	\le
	4 |\G|^{2s-1} \T^{+}_s (\G) \,.
	\end{equation}
	Put $P_j = \{ x ~:~ \rho 2^{j-1} < r_{s\G} (x) \le \rho 2^{j}\} \subseteq \F^*_p$. 
	If (\ref{f:T_2s,T_s}) does not hold, then the possible number of sets $P_j$ does not exceed $L:= s\log |\G|$. 
	By the Dirichlet principle there is $\D = \D_{j_0}$, 
	and a set $P=P_{j_0}$ such that 
	$$
	\T^{+}_{2s} (\G) 
	\le \frac{8}{5} L^4 (2\D)^4 \E^{+} (P) + \mathcal{E} 
	= \T'_{2s} (\G) + \mathcal{E} \,.
	$$
	Indeed, putting $f_i (x) = P_i (x) r_{s\G} (x)$, and using the H\"older inequality, we get 
	$$
	\sum'_{x,y,z} r_{s\G} (x) r_{s\G} (y) r_{s\G} (x+z) r_{s\G} (y+z) 
	\le
	\sum_{i,j,k,l=1}^L\, \sum_{x,y,z} f_i (x) f_j (y) f_k (x+z)  f_l (y+z)
	\le 
	$$
	$$
	\sum_{i,j,k,l=1}^L (\E^{+} (f_i) \E^{+} (f_j) \E^{+} (f_k) \E^{+} (f_l) )^{1/4} 
	=
	\left( \sum_{i=1}^L (\E^{+} (f_i))^{1/4} \right)^4
	\le
	L^3  \sum_{i=1}^L \E^{+} (f_i) \le L^4 \max_i \E^{+} (f_i) \,. 
	$$
	Moreover we always have $|P| \D^2 \le \T^{+}_{s} (\G)$ and $|P| \D \le |\G|^s$. 
	Using Lemma \ref{l:AA_small_energy}, we obtain
	$$
	\E^{+} (P) 
	\le
	C_* \left( \frac{|P|^4}{p} + \frac{|P|^3}{|\G|^{1/2}} \right) \,.
	$$
	Hence
	\begin{equation}\label{tmp:30.03_1}
	\T'_{2s} (\G)
	\le \frac{8}{5}  (16 C_*) L^4 \D^4 \left( \frac{|P|^4}{p} + \frac{|P|^3}{|\G|^{1/2}} \right)
	\le
	\frac{8}{5} (16 C_*) L^4 \left( \frac{|\G|^{4s}}{p} + \frac{|P|^3 \D^4}{|\G|^{1/2}} \right) \,.
	\end{equation}
	Let us consider 
	the second term in (\ref{tmp:30.03_1}).
	Then in view of $|P| \D^2 \le \T^{+}_{s} (\G)$ and $|P| \D \le |\G|^s$, we have 
	$$
	|P|^3 \D^4 = (P \D)^2 P \D^2 \le |\G|^{2s}  \T^{+}_{s} (\G) \,.
	$$
	In other words, by (\ref{f:E_bound}), we get 
	$$
	\T^{+}_{2s} (\G) \le \frac{8}{5} (16 C_*) L^4 \left( \frac{|\G|^{4s}}{p} + |\G|^{2s-1/2} \T^{+}_s (\G) \right)  
	+
	4 |\G|^{2s-1} \T^{+}_s (\G)
	\le
	$$
	$$ 
	\le
	32 C_* s^4 \log^4 |\G| \cdot  \left( \frac{|\G|^{4s}}{p} + |\G|^{2s-1/2} \T^{+}_s (\G) \right) 
	$$
	and inequality  (\ref{f:T_2s,T_s}) is proved. 
	
	Now applying formula  (\ref{f:T_2s,T_s})  successively $(k-1)$ times, we obtain 
	$$
	\T^{+}_{2^k} (\G) \le 2^{4k+6} C_* \log^4 |\G| 
	\cdot 
	\frac{|\G|^{2^{k+1}}}{p} + 16^{k^2} C^{k-1}_* \log^{4(k-1)} |\G| \cdot |\G|^{2^k + \dots + 4 -\frac{(k-1)}{2}} \E^{+} (\G)  
	\le
	$$
	\begin{equation}\label{tmp:30.03_2}
	\le
	2^{4k+6} C_* \log^4 |\G| \cdot \frac{|\G|^{2^{k+1}}}{p} + 16^{k^2} C^{k-1}_* \log^{4(k-1)} |\G| \cdot 
	|\G|^{2^{k+1} -\frac{(k+7)}{2}} \E^{+} (\G) \,.
	\end{equation}
	To get the first term in the last formula we have used our condition 
	$2^{64 k} C^4_*\le |\G|^{}$ to insure that $|\G|^{1/2} \ge 2^{4k+1} C_* \log^4 |\G|$. 
	This completes the proof. 
	$\hfill\Box$
\end{proof}

\begin{remark}
	The condition  
	$2^{64 k} C^{4}_* \le |\G|^{}$ 
	can be dropped but then we will have the multiple $16^{k^2} (C_* \log |\G|)^{k-1}$ in the first term of (\ref{f:T_k_G}). 
\end{remark}

Splitting any $\G$---invariant set onto cosets over $\G$ and applying the norm property of $\T^{+}_l$, we obtain

\begin{corollary}
	Let $\G \subseteq \F_p^*$ be a multiplicative subgroup, and $Q \subseteq \F_p^*$ be a set with  $Q \G = Q$. 
	Then for any $k\ge 2$, 
	$2^{64 k} C^{4}_*\le |\G|^{}$ 
	one has 
	\begin{equation}\label{f:Holder_Q}
	\T^{+}_{2^k} (Q) \le
	2^{4k+6} C_* \log^4 |\G| \cdot \frac{|Q|^{2^{k+1}}}{p} + 16^{k^2} C^{k-1}_* \log^{4(k-1)} |\G| \cdot 
	|\G|^{-\frac{(k+7)}{2}} \E^{+} (\G) |Q|^{2^{k+1}} \,.
	\end{equation}	
	\label{c:Holder_Q}	
\end{corollary}

Let $\G$ be a subgroup of size less than $\sqrt{p}$. 
Considering a particular case $k=2$ of formula (\ref{t:T_k_G}) of Theorem \ref{t:T_k_G} and using $\E^{+} (\G) \ll |\G|^{5/2-c}$, where $c>0$ is an absolute constant (see \cite{s_ineq}), one has   

\begin{corollary}
	Let $\G$ be a multiplicative subgroup, $|\G| \le \sqrt{p}$.
	Then 
	$$
	\T^{+}_4 (\G) \ll \frac{|\G|^8 \log^4 |\G|}{p} + |\G|^{6-c} \,.
	$$
	In particular, 
	$|4\G| \gg |\G|^{2+c}$.  
\end{corollary}

Previous results on $\T^{+}_{k} (\G)$, $|\G| \le \sqrt{p}$ 
with small  $k$ had the form $\T^{+}_{k} (\G) \ll |\G|^{2k-2+c_k}$ with some $c_k>0$, see, e.g., \cite{KS1}. 
The best 
upper bound for
$\T^{+}_3 (\G)$ can be found in \cite{Steinikov_T_3}.

\bigskip 

Now we prove a corollary about exponential sums over subgroups which is parallel to results from \cite{BG}, \cite{BGK}, \cite{Garaev_survey}. 
The difference between the previous 
estimates and Corollary \ref{c:exp_sums} 
is just slightly better constant 
$C$ in (\ref{f:exp_sums'}).

\begin{corollary}
	Let $\G$ be a multiplicative subgroup, $|\G| \ge p^\d$, $\d>0$.
	Then
	\begin{equation}\label{f:exp_sums}
	\max_{\xi \neq 0} |\FF{\G} (\xi)| \ll |\G| \cdot p^{-\frac{\d}{2^{7+ 2\d^{-1}}}} \,.
	\end{equation}
	Further we have a nontrivial upper bound $o(|\G|)$ for the maximum in (\ref{f:exp_sums})  if 
	\begin{equation}\label{f:exp_sums'}	
	\log |\G| \ge \frac{C\log p}{\log \log p} \,,
	\end{equation}
	where $C>2$ is any constant.
	\label{c:exp_sums}
\end{corollary}
\begin{proof}
	We can assume that $|\G| < \sqrt{p}$, say, because otherwise estimate (\ref{f:exp_sums}) is known, see \cite{KS1}.
	By $\rho$ denote the maximum in (\ref{f:exp_sums}). 
	Then by Theorem \ref{t:T_k_G}, a trivial bound $\E^{+}(\G) \le |\G|^3$ 
	and formula (\ref{def:T_k}), we obtain 
	\begin{equation}\label{tmp:31.03_1}
	|\G| \rho^{2^{k+1}} \le p \T_{2^k} (\G) \le  
	2^{4k+6} C_* \log^4 |\G| \cdot |\G|^{2^{k+1}} + 16^{k^2} C^{k-1}_* \log^{4(k-1)} |\G| \cdot 
	|\G|^{2^{k+1} -\frac{k+1}{2}} p \,,
	\end{equation}
	provided  $2^{64k} C^4_* \le |\G|$. 
	Put $k = \lceil 2\log p /\log |\G| + 4 \rceil \le 2/\d +5$. 
	Also, notice that 
	\begin{equation}\label{tmp:k_choice}
	\frac{p \log^{4(k-1)} |\G| }{|\G|^{k/2}} 
	\le 
	1
	\end{equation}
	because $k\ge 2\log p /\log |\G| + 4$ and $p$ is a sufficiently large number.
	Also, since  $|\G| \ge p^\d$, it follows that  $2^{64k} C^4_* \le |\G|$ 
	for sufficiently large $p$. 
	Taking  a power $1/2^{k+1}$ from both parts of (\ref{tmp:31.03_1}), we see 
	in view of (\ref{tmp:k_choice}) 
	that 
	$$
	\rho \ll |\G| \left( |\G|^{-\frac{1}{2^{k+2}}} + |\G|^{-\frac{1}{2^{k+2}}} \right)
	\ll
	|\G|^{1-\frac{1}{2^{k+2}}}
	\ll
	|\G| \cdot p^{-\frac{\d}{2^{7+ 2\d^{-1}}}} \,.
	$$
	To prove the second part of our corollary just notice that the same choice of $k$ gives something nontrivial if $2^{k+2}  \le \eps \log |\G|$ for any $\eps>0$. 
	In other words, it is enough to have 
	$$
	k  + 2 \le \frac{2\log p}{\log |\G|} + 7  \le \log \log |\G| - \log (1/\eps) \,.
	$$
	It means that the inequality  $\log |\G| \ge C \log p /(\log \log p)$ for any $C>2$ is enough. 
	This completes the proof. 
	$\hfill\Box$
\end{proof}


\begin{remark}
	One can improve 
	some constants in the proof (but not the constant $C$ in (\ref{f:exp_sums'})), probably, 
	but we did not make such calculations.
\end{remark}


Now we estimate a "dual"\, quantity $\E^{+}_s (Q)$ for $\G$--invariant set $Q$ (about duality of $\T^{+}_{k/2} (A)$ and  $\E^{+}_k (A)$, see \cite{SS1} and formulae (\ref{tmp:12.04_1--})---(\ref{tmp:12.04_1.5})). 
We give even two bounds and both of them use the Fourier approach.

\begin{theorem}
	Let $\G \subseteq \F_p^*$ be a multiplicative subgroup, and $Q \subseteq \F^*_p$ be a set with $Q \G = Q$
	and $|Q|^2 |\G| \le p^2$.
	Then for $0 \le k$, $2^{64k} C^4_* \le |\G|$ one has 
	$$
	\E^{+}_{2^{k+1}} (Q) 
	\le
	2^{2^{k+2}+3} (\log |Q|)^{2^{k+1}} |Q|^{2^{k+1}}  
	\times 
	$$
	\begin{equation}\label{f:Q_shift}
	\times
	\left(
	2^{4k+6} \log^4 |Q| + 16^{k^2} C^{k-1}_* (\log |Q|)^{4(k-1)} \cdot 
	|\G|^{-\frac{(k+1)}{2}} p  \right) \,.
	\end{equation}
	Further 
	let $k\ge 1$ be such that 
	$|\G|^{\frac{k+2}{2}} \ge |Q| \log^{4k} |Q|$.
	Then 
	\begin{equation}\label{f:Q_shift'}
	\E^{+}_{2^{k+1}} (Q)
	\le (2^8 C_*)^{k+1} |Q|^{2^{k+1}} |\G|^{\frac{1}{2}} 
	\,.
	\end{equation}
	\label{t:Q_shift}
\end{theorem}
\begin{proof}
	We begin with (\ref{f:Q_shift}) and we prove this inequality by induction. 
	For $k=0$ the result is trivial in view of our condition $|Q|^2 |\G| \le p^2$.
	Put $s=2^{k}$, $k\ge 1$.  
	By the Parseval identity and formula  (\ref{f:F_svertka}), we have
	\begin{equation}\label{tmp:12.04_1--}
	\E^{+}_{2s} (Q) = \frac{1}{p^{2s-1}} \sum_{x_1+\dots+x_{2s}=0} |\FF{Q} (x_1)|^2 \dots |\FF{Q} (x_{2s})|^2 
	\le
	\end{equation}
	\begin{equation}\label{tmp:12.04_1-}
	\le
	\frac{2s|Q|^{2} \E^{+}_{2s-1} (Q)}{p}  + \frac{1}{p^{2s-1}} \sum_{x_1+\dots+x_{2s}=0 ~:~ \forall j\,\, x_j\neq 0} |\FF{Q} (x_1)|^2 \dots |\FF{Q} (x_{2s})|^2 
	=
	\end{equation}
	\begin{equation}\label{tmp:12.04_1}
	= 
	\frac{2s|Q|^{2} \E^{+}_{2s-1} (Q)}{p} + \E'_{2s} (Q)\,.
	\end{equation}
	Put $L=\log |Q|$. 
	By the Parseval identity
	$$
	\frac{1}{p^{2s-1}} \sum_{x_1+\dots+x_{2s}=0 ~:~ \forall j\,\, x_j\neq 0} |\FF{Q} (x_1)|^2 \dots |\FF{Q} (x_{2s})|^2
	\le
	$$
	$$
	\le
	\max_{x\neq 0} |\FF{Q} (x)|^2 \cdot \frac{1}{p^{2s-1}} \sum_{x_1+\dots+x_{2s}=0 ~:~ \forall j\,\, x_j\neq 0} |\FF{Q} (x_1)|^2 \dots |\FF{Q} (x_{2s-1})|^2
	\le 
	\max_{x\neq 0} |\FF{Q} (x)|^2 \cdot |Q|^{2s-1} \,.
	$$
	Hence as in the proof of Theorem \ref{t:T_k_G} consider $\rho^2 = \E^{+}_{2s} (Q)/ (8|Q|^{2s-1})$, further, the sets $P_j = \{ x ~:~ \rho 2^{j-1} < |\FF{Q} (x)| \le \rho 2^j \} \subseteq \F_p^*$ and using 
	the Dirichlet principle, we find $\D=\D_{j_0} \ge \rho$ and $P=P_{j_0}$ such that 
	\begin{equation}\label{tmp:12.04_1.5}
	\E'_{2s} (Q) \le \frac{4L^{2s} (2\D)^{4s}}{p^{2s-1}} \T^{+}_{s} (P) \,.	
	\end{equation}
	Clearly, $P\G = P$ (and this is the crucial point of  the proof, actually). 
	Applying Corollary \ref{c:Holder_Q}, we get
	\begin{equation}\label{tmp:31.03_2}
	\E'_{2s} (Q) \le \frac{2^{4s+2} L^{2s} \D^{4s}}{p^{2s-1}} \cdot \left(
	2^{4k+6} \log^4 |\G| \cdot \frac{|P|^{2s}}{p} + 16^{k^2} C^{k-1}_* \log^{4(k-1)} |\G| \cdot 
	|\G|^{-\frac{(k+7)}{2}} \E^{+} (\G) |P|^{2s} \right) \,.
	\end{equation}
	By the Parseval identity, we see that
	\begin{equation}\label{tmp:P_energy-}
	\D^2 |P| \le |Q| p \,.
	\end{equation}
	Whence 
	\begin{equation}\label{tmp:12.04_2}
	\E'_{2s} (Q) \le 2^{4s+2} L^{2s} |Q|^{2s}  \cdot \left(
	2^{4k+6} L^4 + 16^{k^2} C^{k-1}_* L^{4(k-1)} \cdot 
	|\G|^{-\frac{(k+7)}{2}} \E^{+} (\G)  p  \right) \,.
	\end{equation}
	Using a trivial bound $\E^{+} (\G) \le |\G|^3$, we get 
	\begin{equation}\label{tmp:06.05_1}
	\E'_{2s} (Q) 
	\le 2^{4s+2}
	L^{2s} |Q|^{2s}  \cdot \left(
	2^{4k+6} L^4 + 16^{k^2}  C^{k-1}_* L^{4(k-1)} \cdot 
	|\G|^{-\frac{(k+1)}{2}} p  \right)
	\,.
	\end{equation}
	Applying a crude bound
	$\E^{+}_{2s-1} (Q) \le |Q|^{s-1} \E^{+}_s (Q)$, 
	the condition $|Q|^2 |\G| \le p^2$,  
	and induction assumption, 
	we get
	$$
	\frac{2s|Q|^{2} \E^{+}_{2s-1} (Q)}{p} \le \frac{2s |Q|^{s+1} \E^{+}_s (Q)}{p}
	\le
	$$
	$$
	\le
	\frac{2s |Q|^{s+1}}{p} \cdot L^s |Q|^{s}  \cdot 2^{2s+3} \left(
	2^{4k+2} L^4 + 16^{(k-1)^2} C^{k-2}_* L^{4(k-2)} \cdot 
	|\G|^{-\frac{k}{2}} p  \right) 
	\le
	$$
	$$		
	\le
	2^{4s+2} L^{2s} |Q|^{2s}  \cdot \left(
	2^{4k+6} L^4 + 16^{k^2} C^{k-1}_* L^{4(k-1)} \cdot 
	|\G|^{-\frac{(k+1)}{2}} p  \right)
	\,.
	$$
	Hence combining the last estimate with (\ref{tmp:06.05_1}), we derive 
	$$
	\E^{+}_{2^{k+1}} (Q)
	\le
	2^{2^{k+2}+3} L^{2^{k+1}} |Q|^{2^{k+1}}  \cdot \left(
	2^{4k+6} L^4 + 16^{k^2} C^{k-1}_* L^{4(k-1)} \cdot 
	|\G|^{-\frac{(k+1)}{2}} p  \right)
	$$
	and thus we have obtained (\ref{f:Q_shift}).

	To get (\ref{f:Q_shift'}), put $l=2^{k-1}$, $k\ge 1$ and consider $\E^{+}_{4l} (Q)$.
	Further
	define $g(x) = r^l_{Q-Q} (x)$ and notice that $\FF{g} (\xi) \ge 0$, $\FF{g} (0) = \E^{+}_l (Q)$. 
	Moreover, taking the Fourier transform and using the Dirichlet principle, we get 
	\begin{equation}\label{tmp:12.04_1'}
	\E^{+}_{4l} (Q) 
	= \frac{1}{p^3} \sum_{x,y,z} \FF{g} (x) \FF{g} (y) \FF{g} (x+z) \FF{g}(y+z)
	=
	\frac{\E^{+} (\FF{g})}{p^3} 
	\le 
	\frac{4 \E^{+}_l (Q) \E^{+}_{3l} (Q)}{p} 
	+ \frac{4 L^4 (2\o)^4}{p^3} \E^{+} (G) \,,
	\end{equation}
	where $G=\{ \xi ~:~ \omega < \FF{g} (\xi) \le 2 \omega \} \subseteq \F_p^*$, and 
	$\omega \ge 2^{-3} \E^{+}_{4l} (Q) |Q|^{-3l} := \rho_*$
	because the sum over $\FF{g}(\xi) < \rho_*$ by formula (\ref{f:inverse}) does not exceed 
	$$	
	\frac{4\rho_*}{p^3} \cdot \sum_{x,y,z} \FF{g} (y) \FF{g} (x+z) \FF{g}(y+z)
	=
	4\rho_* g^3 (0) =  4 \rho_* |Q|^{3l} \,.
	$$
	Further in view of the Parseval identity, we see that  
	\begin{equation}\label{tmp:P_energy}
	\o^2 |G| \le \sum_{\xi \in G} \FF{g} (\xi)^2 \le  p \E^{+}_{2l} (Q) 
	\,, 
	\end{equation}
	and by formula (\ref{f:inverse}) 
	\begin{equation}\label{tmp:P_energy'}
	\o |G| \le  \sum_{\xi \in G} \FF{g} (\xi) = p g(0) = p |Q|^l \,.
	\end{equation}
	Clearly, $G$ is $\G$--invariant set (again it is the crucial point of the proof). 
	Further 
	returning to (\ref{tmp:12.04_1'})
	and applying 
	Lemma \ref{l:AA_small_energy}, 
	we see that 
	$$
	\E^{+}_{4l} (Q) \le
	\frac{4 \E^{+}_l (Q) \E^{+}_{3l} (Q)}{p} + \frac{2^6 L^4 \o^4}{p^3} \E^{+} (G) 
	\le
	\frac{4\E^{+}_l (Q) \E^{+}_{3l} (Q)}{p} + \frac{2^6 C_* L^4 \o^4}{p^3}
	\left( \frac{|G|^4}{p} + \frac{|G|^3}{|\G|^{1/2}} \right) 
	=
	$$
	$$
	=
	\frac{4\E^{+}_l (Q) \E^{+}_{3l} (Q)}{p} + \E'_{4l} (Q)	\,.
	$$  
	Applying (\ref{tmp:P_energy}), (\ref{tmp:P_energy'}), we get 
	$$
	\E'_{4l} (Q) \le 2^6 C_* L^4 |Q|^{4l} + \frac{2^6 C_* L^4 (\o|G|)^2 \o^2 |G|}{|\G|^{1/2} p^3}
	\le
	2^6 C_* L^4 |Q|^{4l} + 2^6 C_* L^4 |Q|^{2l} \E^{+}_{2l} (Q) |\G|^{-1/2} \,.
	$$
	It follows that
	\begin{equation}\label{f:E_4l_ind}
	\E^{+}_{4l} (Q) \le \frac{4\E^{+}_l (Q) \E^{+}_{3l} (Q)}{p}
	+ 
	2^6 C_* L^{4} |Q|^{2l} \E^{+}_{2l} (Q) \left( \frac{|Q|^{2l}}{\E^{+}_{2l} (Q)} +  \frac{1}{|\G|^{1/2}} \right) \,.
	\end{equation}
	Further estimating the first term of (\ref{f:E_4l_ind}) very roughly as 
	$$
	\frac{\E^{+}_l (Q) \E^{+}_{3l} (Q)}{p} \le \frac{|Q|^{l+1} \E^{+}_{3l} (Q)}{p} \le \frac{|Q|^{2l+1} \E^{+}_{2l} (Q)}{p} \,,
	$$
	we get in view of our condition $|Q|^2 |\G| \le p^2$ that this term is less than $L^{4} |Q|^{2l} \E^{+}_{2l} (Q) |\G|^{-1/2}$.
	Whence
	\begin{equation}\label{f:E_4l_ind-}
	\E^{+}_{4l} (Q) \le 2^7 C_* 
	L^{4} |Q|^{2l} \E^{+}_{2l} (Q) \left( \frac{|Q|^{2l}}{\E^{+}_{2l} (Q)} +  \frac{1}{|\G|^{1/2}} \right) \,.
	\end{equation}
	Notice that the term $\frac{|Q|^{2l}}{\E^{+}_{2l} (Q)} +  \frac{1}{|\G|^{1/2}} \le 2$. 
	Applying bound (\ref{f:E_4l_ind}) exactly $0\le s \le k$ times, where $s$ is the maximal number (if it exists) such that the second term $\frac{1}{|\G|^{1/2}}$ in formula (\ref{f:E_4l_ind-}) dominates,
	we obtain 
	\begin{equation}\label{f:19.04.2017_1}
	\E^{+}_{2^{k+1}} (Q) \le (2^8 C_*)^s 
	L^{4s} |\G|^{-s/2} |Q|^{2^{k} + \dots + 2^{k-s+1}} 
	\E^{+}_{2^{k-s+1}} (Q) \left( \frac{|Q|^{2^{k-s+1}}}{\E^{+}_{2^{k-s+1}} (Q)} +  \frac{1}{|\G|^{1/2}} \right)
	\,.
	\end{equation}
	Now by the definition of $s$, we see that the first  term in (\ref{f:19.04.2017_1}) dominates.
	Whence, using (\ref{f:E_4l_ind}), (\ref{f:E_4l_ind-})  one more time (if $s<k$), we 
	get 
	$$
	\E^{+}_{2^{k+1}} (Q) \le 2(2^8 C_*)^{s} 
	L^{4s} |\G|^{-s/2} |Q|^{2^{k+1} - 2^{k-s+1}} \cdot |Q|^{2^{k-s+1}}
	=
	$$
	\begin{equation}\label{tmp:06.05_2} 
	=
	2(2^8 C_*)^{s}  L^{4s} |\G|^{-s/2} |Q|^{2^{k+1}} \,.
	\end{equation}
	From the assumption $|\G|^{\frac{k+2}{2}} \ge |Q| \log^{4k} |Q|$, it follows that $|\G|\ge  |Q|^{2/(k+2)} \log^{8k/(k+2)} |Q|$. 
	Hence bound (\ref{tmp:06.05_2}) is much better than (\ref{f:Q_shift'}) if $s<k$.
	If $s=k$, then by the same calculations, we 
	derive 
	$$
	\E^{+}_{2^{k+1}} (Q) \le (2^8 C_*)^{k}
	L^{4k} |\G|^{-k/2} \E^{+}_2 (Q) |Q|^{2^{k+1}-2} \,.
	$$
	Since $|Q|^2 |\G| \le p^2$ by Lemma \ref{l:AA_small_energy}, it follows that 
	$\E^{+} (Q) \le 2C_* |Q|^3 / |\G|^{1/2}$
	and hence
	$$
	\E^{+}_{2^{k+1}} (Q) \le (2^8 C_*)^{k+1} L^{4k} |\G|^{-(k+1)/2} |Q|^{2^{k+1}+1} \,.
	$$
	Further by the choice of $k$, namely, 
	$|\G|^{\frac{k+2}{2}} \ge |Q| \log^{4k} |Q|$ 
	we see that the last bound is better than (\ref{f:Q_shift'}). 
	Finally, if 
	$s=0$, then by definition $\E^{+}_{2^{k}} (Q) \le |Q|^{2^{k}} |\G|^{1/2}$ 
	and hence $\E^{+}_{2^{k+1}} (Q) \le |Q|^{2^{k+1}} |\G|^{1/2}$. 
	This completes the proof. 
	$\hfill\Box$
\end{proof}

\begin{remark}
	From the second part of the arguments above one can derive explicit  bounds for the energies $\E^{+}_{s} (Q)$ for small $s$.
	For example, 
	$$
	\E_{4} (Q) \ll \frac{|Q|^2 \E_3 (Q)}{p} +  (\log |\G|)^{4} |Q|^4 + (\log |\G|)^{4} |Q|^2 \E(Q) |\G|^{-1/2}  \,.
	$$   
\end{remark}

Now we obtain an uniform upper bound for size of the intersection of an additive shift of any $\G$--invariant set. 
Our bound (\ref{f:Q_cap_Q+x'}) is especially effective if sizes of $Q_1$,$Q_2$ are comparable with size of $\G$, 
namely, $|Q_1|,|Q_2| \ll |\G|^C$, $C$ is an absolute constant (which can be large). 
In this case the number $k$ below is a constant as well.

\begin{corollary}
	Let $\G \subseteq \F_p^*$ be a multiplicative subgroup, 
	$|\G| \ge p^\d$, $\d>0$, and $Q_1,Q_2 \subseteq \F^*_p$ be two sets with $Q_1 \G = Q_1$, $Q_2 \G = Q_2$,
	$|Q_1|^2 |\G| \le p^2$, $|Q_2|^2 |\G| \le p^2$.
	Put $Q=\max\{|Q_1|, |Q_2| \}$. 
	Then for any $x\neq 0$, one has
	\begin{equation}\label{f:Q_cap_Q+x}	
	|Q_1\cap (Q_2+x)| \ll  \sqrt{|Q_1||Q_2|} \log Q \cdot  p^{-\frac{\d}{2^{7+2\d^{-1}}}}   \,.
	\end{equation}
	Further choose $k\ge 1$ such that 
	$|\G|^{\frac{k+2}{2}} \ge Q \log^{4k} Q$.
	Then  for an arbitrary $x\neq 0$ the following holds
	\begin{equation}\label{f:Q_cap_Q+x'}	
	|Q_1\cap (Q_2+x)| \ll  \sqrt{|Q_1||Q_2|} \cdot |\G|^{-\frac{1}{4} \cdot 2^{-k}} \,.
	\end{equation}
	\label{c:Q_cap_Q+x}
\end{corollary}
\begin{proof}
	From the conditions $|Q_1|^2 |\G| \le p^2$, $|Q_2|^2 |\G| \le p^2$, it follows that $|\G| \le p^{2/3}$. 
	Put $L = \log Q$. 
	On the one hand, applying 
	the Cauchy--Schwarz inequality, we obtain 
	$$
	\sum_{y} r^{2^{k+1}}_{Q_1 - Q_2} (y) 
	\le (\E^{+}_{2^{k+1}} (Q_1))^{1/2} (\E^{+}_{2^{k+1}} (Q_2))^{1/2} \,.
	$$ 
	On the other hand,  by formula (\ref{f:Q_shift}) of Theorem \ref{t:Q_shift} and $\G$--invariance of $Q_1$, $Q_2$, we have 
	$$
	|\G| |Q_1 \cap (Q_2+x)|^{2^{k+1}} 
	\le 
	\sum_{y} r^{2^{k+1}}_{Q_1 - Q_2} (y) 
	\le
	$$
	$$ 
	\le
	2^{2^{k+2}+3} L^{2^{k+1}} (|Q_1| |Q_2|)^{2^{k}} 
	\left(
	2^{4k+6} \log^4 |Q| + 16^{k^2} C^{k-1}_* L^{4(k-1)} \cdot 
	|\G|^{-\frac{(k+1)}{2}} p  \right) \,,
	$$
	provided $2^{64k} C^4_* \le |\G|$.
	As in Corollary \ref{c:exp_sums} choosing $k = \lceil 2 \log p /\log |\G| + 4 \rceil \le 2/\d+5$ and 
	applying an analogue of 
	(\ref{tmp:k_choice}) which 
	holds 
	for large $p$, namely,
	$$
	\frac{p \log^{4(k-1)} |Q|}{|\G|^{k/2}} \ll 1 
	$$
	we obtain 
	$$
	|Q_1\cap (Q_2+x)| \ll L \sqrt{|Q_1||Q_2|} \cdot (|\G|^{-1/2^{k+2}} + |\G|^{-1/2^{k+2}})
	\ll
	L  \sqrt{|Q_1||Q_2|} |\G|^{-1/2^{k+2}} 
	\ll 
	$$
	$$ 
	\ll
	L\sqrt{|Q_1||Q_2|} p^{-\frac{\d}{2^{7+2\d^{-1}}}} 
	$$
	and it easy to insure that inequality $2^{64k} C^4_* \le |\G|$ takes place for sufficiently large $p$.

	To derive (\ref{f:Q_cap_Q+x'}), we just use the second formula  (\ref{f:Q_shift'}) of Theorem \ref{t:Q_shift} and the previous calculations.
	This completes the proof. 
	$\hfill\Box$
\end{proof}

\bigskip

\begin{remark}
	It is known, see, e.g., \cite{KS1} that if $\G \subseteq \F_p^*$ is a multiplicative subgroup with $|\G| < p^{3/4}$, then for any $x\neq 0$ one has $|\G\cap (\G+x)| \ll |\G|^{2/3}$ and this bound it tight in some regimes. 
	One can extend this to larger $\G$--invariant sets and obtain a lower bound of a comparable  quality. It gives  a lower estimate in (\ref{f:Q_cap_Q+x}).

	Indeed, let $\G \subseteq \F_p^*$ be a multiplicative subgroup with $|\G| < p^{1/2}$. 
	Consider $R=R[\G]$ and $Q=Q[\G]$.  
	It was proved in \cite{Shkr_tripling} that $|R| \gg |\G|^2 / \log |\G|$
	and one can check that $R = 1-R$, see, e.g., \cite{MPR-NRS}.
	Finally, the set $Q$ is $\G$--invariant and it is easy to check \cite{s_diff} that $|Q| \le |\G|^3$. 
	Whence 
	$$
	|Q \cap (1-Q)| \ge |R| \gg \frac{|\G|^2}{\log |\G|} 
	\gg 
	\frac{|Q|^{2/3}}{\log |Q|} \,.
	$$

	Also, notice that if $|\G|<p^{1/2}$ and $|Q [\G]|^2 |\G| \le p^2$, then $|Q [\G]| \gg |\G|^{2+c}$ for some $c>0$, see the first part of Corollary \ref{c:superquadratic} from the next section.   
	\label{r:2/3}
\end{remark}

\bigskip 

Corollary 
\ref{c:Q_cap_Q+x}
gives a nontrivial upper bound  for the common additive energy of an arbitrary invariant set and {\it any} subset of $\F_p$.

\begin{corollary}
	Let $\G \subseteq \F_p^*$ be a multiplicative subgroup, 
	$|\G| \ge p^\d$, $\d>0$, and $Q \subseteq \F^*_p$ be a set with  $Q \G = Q$, $|Q|^2 |\G| \le p^2$. 
	Then for any set $A\subseteq \F_p$, 
	one has
	\begin{equation}\label{f:E(Q,A)}	
	\E^{+} (A,Q) \ll |Q| |A|^2 \cdot p^{-\frac{\d}{2^{7+2\d^{-1}}}} \log |Q| + |A| |Q| \,.
	\end{equation}
	Further for an arbitrary $\a \neq 0$ the following holds 
	\begin{equation}\label{f:E(Q,A)+}	
	\E^{\times} (A,Q+\a) \ll |Q| |A|^2 \cdot p^{-\frac{\d}{2^{7+2\d^{-1}}}} \log |Q| + |A| |Q| \,.
	\end{equation}
	In particular, 
	\begin{equation}\label{f:E(Q,A)'}	
	|A+Q| \gg |Q| \cdot \min \{ |A|, p^{\frac{\d}{2^{7+2\d^{-1}}}} \log^{-1} |Q| \} \,,
	\end{equation}
	and
	\begin{equation}\label{f:E(Q,A)'+}	
	|A(Q+\a)| \gg |Q| \cdot \min \{ |A|, p^{\frac{\d}{2^{7+2\d^{-1}}}} \log^{-1} |Q| \} \,.
	\end{equation}	
	Further if 
	$k\ge 1$ is chosen as  
	$|\G|^{\frac{k+2}{2}} \ge |Q| \log^{4k} |Q|$, 
	then one can replace the quantity\\ 
	$p^{\frac{\d}{2^{7+2\d^{-1}}}} \log^{-1} |Q|$ above by $|\G|^{-\frac{1}{4} \cdot 2^{-k}}$. 
	\label{c:E(Q,A)}
\end{corollary}
\begin{proof}
	Inequalities  (\ref{f:E(Q,A)'}), (\ref{f:E(Q,A)'+}) follow from (\ref{f:E(Q,A)}), (\ref{f:E(Q,A)+})  via the Cauchy--Schwarz inequality, so it is enough to obtain the required  upper bound for the additive energy of $A$ and $Q$ and for  the multiplicative energy of $A$ and $Q+\a$. 
	By Corollary \ref{c:Q_cap_Q+x}, we have 
	$$
	\E^{+} (A,Q) = \sum_x r_{A-A} (x) r_{Q-Q} (x) 
	= |A||Q| + \sum_{x\neq 0} r_{A-A} (x) r_{Q-Q} (x)
	\ll
	|A||Q| + |Q| |A|^2 \cdot p^{-\frac{\d}{2^{7+2\d^{-1}}}} \log |Q|
	$$
	as required. 
	Similarly 
	$$
	\E^{\times} (A,Q+\a) 
	\ll |A||Q| + \sum_{x\neq 0,1}  r_{A/A} (x) r_{(Q+\a)/(Q+\a)} (x)
	\ll
	|A||Q| + |Q| |A|^2 \cdot p^{-\frac{\d}{2^{7+2\d^{-1}}}} \log |Q|
	$$
	because in view of Corollary \ref{c:Q_cap_Q+x} one has
	$$
	r_{(Q+\a)/(Q+\a)} (x) = |Q\cap (xQ + \a(x-1))| 
	\ll
	|Q| \cdot p^{-\frac{\d}{2^{7+2\d^{-1}}}} \log |Q| \,.
	$$
	So, we have obtained bounds (\ref{f:E(Q,A)})---(\ref{f:E(Q,A)'+}) with $p^{\frac{\d}{2^{7+2\d^{-1}}}} \log^{-1} |Q|$ and to replace it by $|\G|^{-\frac{1}{4} \cdot 2^{-k}}$ one should use the second part of Corollary \ref{c:Q_cap_Q+x}. 
	This completes the proof. 
	$\hfill\Box$
\end{proof}

\bigskip

From (\ref{f:E(Q,A)'}) one can obtain that for any multiplicative subgroup $\G \subseteq \F_p^*$ there is $N$ such that $N\G = \F_p$ and $N\ll \d^{-1} 4^{\d^{-1}}$.
The results of comparable quality were obtained in \cite{GK}.

\section{The proof of the main result}
\label{sec:small_prod}


In this section we obtain an upper bound for $\T^{+}_k (A)$  (see Theorem \ref{t:T_k_G_M})
and an upper bound for $\E^{+}_k (A)$ (see Theorem \ref{t:QG}) 
in the case when 
size of the product set $AB$ is small comparable to $A$, where  $B$ is a sufficiently large set.
From the last result we derive our quantitative asymmetric sum--product Theorem \ref{t:T_k_G_M_intr} from the introduction. 
Let us begin with an upper bound for $\T^{+}_k (A)$.

\begin{theorem}
	Let $A,B \subseteq \F_p$  be sets, $M\ge 1$ be a real number and 
	$|AB| \le M|A|$. 
	Then for any $k\ge 2$, 
	$2^{16 k} M^{2^{k+1}} C^2_* \log^8 |A| \le |B|^{}$,  
	one has 
	\begin{equation}\label{f:T_k_G_M}
	\T^{+}_{2^k} (A) \le	
	2^{4k+6} C_* \log^4 |A| \cdot \frac{M^{2^{k}}|A|^{2^{k+1}}}{p} +  16^{k^2} C_*^{k-1} M^{2^{k+1}}  \log^{4(k-1)} |A| \cdot 
	|A|^{2^{k+1} - 4} |B|^{-\frac{(k-1)}{2}} \E^{+} (A) \,.
	\end{equation}
	\label{t:T_k_G_M}
\end{theorem}
\begin{proof}
	We have $B \neq \{ 0\}$ by the condition $2^{16 k} M^{2^{k+1}} C^2_* \log^8 |A| \le |B|^{}$, say. 
	We apply the arguments and the notation of the  proof of Theorem \ref{t:T_k_G}.
	Fix any $s \ge 2$ and put $L:= s \log |A|$. 
	Our intermediate aim is to prove 
	\begin{equation}\label{f:T_2s,T_s_M}
	\T^{+}_{2s} (A) \le C s^4 M^{2s} \log^{4} |A| \cdot  \left( \frac{|A|^{4s}}{p} + \frac{|A|^{2s}}{\sqrt{|B|}} \T^{+}_s (A) \right) \,,
	\end{equation}
	where $C = 2^5 C_*$.
	As in the proof of Theorem \ref{t:T_k_G}, we get  
	$$
	\T^{+}_{2s} (A) 
	\le \frac{8}{5} L^4 (2\D)^4 \E^{+} (P) + \mathcal{E} \,,
	$$
	where 
	\begin{equation}\label{f:E_bound'}
	\mathcal{E} \le 4 |A|^{2s-1} \T^{+}_s (A)
	\end{equation}
	further 
	$\D> \T^{+}_{2s} (A)/(8|A|^{3s})$ is a real number and 
	$P = \{ x ~:~ \D< r_{sA} (x) \le 2\D \} \subseteq \F^*_p$. 
	Moreover, we always have $|P| \D^2 \le \T^{+}_{s} (A)$.

	To proceed as in the proof of Theorem \ref{t:T_k_G} we need to estimate $|PB|$. 
	Observe that for any $x\in PB$ the following holds $r_{sAB} (x) \ge \D$. 
	Thus, we have 
	\begin{equation}\label{f:PB}
	|PB| \D \le
	\sum_{x\in PB} r_{sAB} (x) \le |AB|^s \le M^s |A|^s \,.
	\end{equation}
	Hence using Lemma \ref{l:AA_small_energy}, we obtain
	$$
	\E^{+} (P) 
	\le C_* \left( 
	\frac{M^{2s} |A|^{2s} |P|^2}{\D^2 p} 
	+ \frac{M^{3s/2} |A|^{3s/2} |P|^{3/2}}{\D^{3/2} |B|^{1/2}} \right) \,.
	$$
	Hence in view of estimate (\ref{f:E_bound'}), combining with  $|P|\D \le |A|^s$ and $|P| \D^2 \le \T^{+}_{s} (A)$, we get
	$$
	\T^{+}_{2s} (A) \le \frac{8}{5} (16 C_*) L^4 \D^4 \left( 
	\frac{M^{2s} |A|^{2s} |P|^2}{\D^2 p} 
	+ \frac{M^{3s/2} |A|^{3s/2} |P|^{3/2}}{\D^{3/2} |B|^{1/2}} \right) + 4 |A|^{2s-1} \T^{+}_s (A)
	=
	$$
	$$
	=
	\frac{8}{5} (16 C_*)  L^4 \left( \frac{M^{2s} |A|^{2s} |P|^2 \D^2}{p} + \frac{M^{3s/2} |A|^{3s/2} |P|^{3/2} \D^{5/2}}{|B|^{1/2}} \right) 
	+ 4 |A|^{2s-1} \T^{+}_s (A)
	\le
	$$
	$$
	\le
	\frac{8}{5} (16 C_*)  L^4 \left( \frac{M^{2s} |A|^{4s}}{p} + \frac{M^{3s/2} |A|^{3s/2} (|P| \D^2) (|P| \D)^{1/2}}{|B|^{1/2}} \right) 
	+ 4 |A|^{2s-1} \T^{+}_s (A)
	\le
	$$
	\begin{equation*}\label{tmp:30.03_1_M}
	\le 
	32 C_*  L^4 \left( \frac{M^{2s} |A|^{4s}}{p} + \frac{M^{3s/2} |A|^{2s} \T^{+}_{s} (A)}{|B|^{1/2}} \right) 
	\end{equation*}
	and inequality  (\ref{f:T_2s,T_s_M}) is proved. 
	Here we have used a trivial inequality $|B|^{1/2} \le |A|$ which follows from 
	$|B| \le |AB| \le M|A| \le |B|^{1/2} |A|$ because $M^2 \le 2^{16 k} M^{2^{k+1}} C^2_* \log^8 |A| \le |B|^{}$.

	Now applying formula  (\ref{f:T_2s,T_s_M})  successively $(k-1)$ times, we obtain 
	$$
	\T_{2^k} (A) \le 2^{4k+6} C_* \log^4 |A| 
	\cdot 
	\frac{M^{2^{k}}|A|^{2^{k+1}}}{p} + 16^{k^2} M^{2^{k+1}} C_*^{k-1} \log^{4(k-1)} |A| \cdot |A|^{2^k + \dots + 4} |B|^{-\frac{(k-1)}{2}} \E^{+} (A)  
	\le
	$$
	\begin{equation}\label{tmp:30.03_2_M}
	\le 
	2^{4k+6} C_* \log^4 |A| \cdot \frac{M^{2^{k}}|A|^{2^{k+1}}}{p} +  16^{k^2} M^{2^{k+1}} C_*^{k-1} \log^{4(k-1)} |A| \cdot 
	|A|^{2^{k+1} - 4} |B|^{-\frac{(k-1)}{2}} \E^{+} (A) \,.
	\end{equation}
	To get the first term in we have used our condition $2^{16 k} M^{2^{k+1}} C^2_* \le |B|^{}$
	to insure that $|B|^{1/2} \ge 2^{4k+1} C_* M^{2^k} \log^4 |A|$.
	This completes the proof. 		
	$\hfill\Box$
\end{proof}


\begin{remark}
	It is easy to see that instead of the assumption  $|AB| \ll |A|$ we can assume a weaker condition  $|A^s \cdot \Delta_s (B)| \ll |A|^s$,
	$1<s \le 2^{k-1}$, 
	see formula (\ref{f:PB}).
\end{remark}


The same arguments work in the case of real numbers. 
In this 
situation 
we have no the characteristic  $p$ and hence we have no any restrictions on the parameter $k$.

\begin{theorem}
	Let $A,B \subset \R$  be finite sets, $M\ge 1$ be a real number and 
	$|AB| \le M|A|$.
	Then for any $k\ge 2$, 
	one has 
	\begin{equation}\label{f:T_k_G_MR}
	\T^{+}_{2^k} (A) \le
	16^{k^2} C_*^{k-1} M^{\frac{3}{2} (2^{k}-1)}  \log^{4(k-1)} |A| \cdot |A|^{2^{k+1}-1} |B|^{-\frac{k}{2}} \,.	 
	\end{equation}
	\label{t:T_k_G_MR}
\end{theorem}

\begin{corollary}
	Let $A \subset \R$  be a finite set, $M\ge 1$ be a real number and 
	$|AA| \le M|A|$ or $|A/A| \le M|A|$. 
	Then for any $k\ge 2$, one has
	\begin{equation}\label{f:M_cor}
	|2^k A| \gg_k |A|^{1+\frac{k}{2}} M^{-\frac{3}{2} (2^{k}-1)} \cdot \log^{-4(k-1)} |A| \,.
	\end{equation}
\end{corollary}

Bounds of such a sort were obtained in \cite{K_mult} by another method. 
The best results concerning lower bounds for multiple sumsets $kA$, $k \to \infty$ of sets $A$ with small product/quotient 	set can be found in \cite{BC}.

\bigskip

To obtain an analogue of Theorem \ref{t:Q_shift} for sets with $|AA|\ll |A|$ we cannot use the same arguments as 
in section \ref{sec:proof} 
because the spectrum is not an invariant set in this case. 
Moreover, in $\R$ there is an additional difficulty with using Fourier transform : the dual group of $\R$ does not coincide with $\R$ of course. 
That is why we suggest another  method which works in "physical space"\, but not in the dual group.

To formulate our main result about $\E^{+}_k (Q)$ for sets $Q$ with small product $Q\G$ for some relatively large set $\G$
we need some notation.  
Let us write $Q^{(k)} = |Q\G^{k-1}|$ for $k\ge 1$ and $Q^{(k)} = |Q|$ for $k=-1$.

\begin{theorem}
	Let $\G,Q \subseteq \F_p$ be two sets, and $k\ge 0$ be an integer.
	Suppose that $|Q\G^{k+1}| |Q\G^k| |\G| \le p^2$, 
	further 
	$Q^{(k)} |\G| \le p$, 
	and $M = |Q \G^{k+1}|/|Q|$. 
	Then 
	either 
	\begin{equation}\label{f:QG}
	\E^{+}_{2^{k+1}} (Q) 
	\le 
	(M^{2^k+1} 2^{3k+1} C^{(k+4)/4}_* \log^{k} Q^{(k)})  \cdot
	|Q|^{2^{k+1}+1} |\G|^{-k/8-1/2} 
	\end{equation}
	or 
	$$
	\E^{+}_{2^{k+1}} (Q)  \le 2 (Q^{(k)})^{2^{k+1}} \,.
	$$
	In particular, if we choose $k$  such that 
	$|\G|^{k/8+1/2} \ge |Q| \cdot  M^{2^k+1} 2^{3k+1} C^{(k+4)/4}_* \log^{k} Q^{(k)}$, then  
	\begin{equation}\label{f:QG'}
	\E^{+}_{2^{k+1}} (Q) \le 2 (Q^{(k)})^{2^{k+1}}  \,.
	\end{equation}
	\label{t:QG}
\end{theorem}
\begin{proof}
	Without loss  of generality one can assume that $0\notin \G$.  
	Fix an integer $l\ge 1$ and prove that either 
	\begin{equation}\label{f:iteration_5.2}
	\E^{+}_{5l/2} (Q) 
	\le 8 C^{1/4}_* \log^{} |Q| \cdot |Q|^{l/2} \E^{+}_{2l} (Q\G) |\G|^{-1/8} 
	\end{equation}
	or
	\begin{equation}\label{f:iteration_5.2'}
	\E^{+}_{5l/2} (Q) \le 2|Q|^{5l/2} \,.
	\end{equation}
	Put $g(x)  = r^l_{Q-Q} (x)$, $L = \log |Q|$
	and $\E'_{5l/2} (Q) = \E^{+}_{5l/2} (Q) - |Q|^{5l/2} \ge 0$. 
	We will assume below that $\E'_{5l/2} (Q) \ge 2^{-1} \E^{+}_{5l/2} (Q)$ because otherwise 
	we obtain (\ref{f:iteration_5.2'}) immediately. 
	Using the Dirichlet principle, we find a set $P$ and a positive number $\D$ such that $P = \{ x ~:~ \D < g(x) \le 2\D \} \subseteq \F_p^*$ and 
	$$
	\E'_{5l/2} (Q) \le L \sum_{x \in P} r^{5l/2}_{Q-Q} (x) \,. 
	$$  
	Applying Corollary \ref{c:change_QG}, we obtain	
	\begin{equation*}\label{tmp:30.04_1}
	\E'_{5l/2} (Q) \le L (2\D)^{3/2} \sum_{x \in P} r^{l}_{Q-Q} (x)
	\le
	3 C^{1/4}_* L \D^{3/2} |Q|^{l/2} (\E^{+}_{2l} (Q\G))^{1/4} \left( \frac{|P|^4}{p} + \frac{|P|^3}{|\G|^{1/2}}\right)^{1/4} 
	\le 
	\end{equation*}
	$$
	\le
	3 C^{1/4}_* L  |Q|^{l/2} (\E^{+}_{2l} (Q\G))^{1/4} \left( \frac{\D^6 |P|^4}{p} + \frac{\D^6 |P|^3}{|\G|^{1/2}}\right)^{1/4} \,.
	$$
	We have $\D |P| \le \E^{+}_{l} (Q)$, $\D^2 |P| \le \E^{+}_{2l} (Q)$ and hence $\D^6 |P|^4 \le (\E^{+}_{2l} (Q))^2 (\E^{+}_{l} (Q))^2$.
	It follows that
	$$
	\E'_{5l/2} (Q) 
	\le
	3 C^{1/4}_* L  |Q|^{l/2} (\E^{+}_{2l} (Q\G))^{1/4}  
	\left( \frac{(\E^{+}_{2l} (Q))^2 (\E^{+}_{l} (Q))^2}{p} + \frac{(\E^{+}_{2l} (Q))^3}{|\G|^{1/2}}\right)^{1/4} \,.
	$$
	To prove that the first term $\frac{(\E^{+}_{2l} (Q))^2 (\E^{+}_{l} (Q))^2}{p}$ is less than $\frac{(\E^{+}_{2l} (Q))^{3}}{|\G|^{1/2}}$, we need to check that 
	$$
	(\E^{+}_{l} (Q))^2 |\G|^{1/2} \le \E^{+}_{2l} (Q) p \,.
	$$
	But using the H\"older inequality, 
	we see that the required estimate follows from 
	$$
	(\E^{+}_{l} (Q))^2 |\G|^{\frac{1}{2}} \le (\E^{+}_{2l} (Q))^{\frac{2(l-1)}{2l-1}}
	|Q|^{\frac{4l}{2l-1}} |\G|^{\frac{1}{2}}  \le \E^{+}_{2l} (Q) p 
	$$
	or, in other words, from
	\begin{equation}\label{03.05.2017_2}
	|Q|^{4l} |\G|^{(2l-1)/2} \le \E^{+}_{2l} (Q) p^{2l-1} \,.
	\end{equation}
	Finally, we can suppose that for any $s\ge 2$ one has, say,  
	$$
	\E^{+}_s (Q) \ge |Q|^{s+1} |\G|^{-\frac{1}{8} \log s - \frac{1}{2}} 
	$$
	because otherwise estimate (\ref{f:QG}) follows easily. 
	In view of our assumption  $|Q| |\G| \le p$, we obtain
	$$
	|Q|^{2l-1}|\G|^{\frac{1}{8} \log 2l + l} \le p^{2l-1} |\G|^{1 + \frac{1}{8} \log 2l - l} \le p^{2l-1} 
	$$
	and hence (\ref{03.05.2017_2}) takes place for $l\ge 2$. 
	For $l=1$ see calculations below.
	Hence under this assumption and the inequality  $\E'_{5l/2} (Q) \ge 2^{-1} \E^{+}_{5l/2} (Q)$, we have 
	$$
	\E^{+}_{5l/2} (Q) \le 8 C^{1/4}_* \log^{} |Q| \cdot |Q|^{l/2} \E^{+}_{2l} (Q\G) |\G|^{-1/8} 
	$$
	and we have proved (\ref{f:iteration_5.2}). 
	Trivially, it implies that
	$$
	\E^{+}_{4l} (Q) \le 8 C^{1/4}_* \log^{} |Q| \cdot |Q|^{2l} \E^{+}_{2l} (Q\G) |\G|^{-1/8} 
	$$
	and subsequently using  this bound, we obtain
	$$
	\E^{+}_{2^{k+1}} (Q) \le (2^{3k} C^{k/4}_* \log^{k}  |Q\G^{k-1}|)  \cdot M^{2^{k-1} + \dots + 2} |Q|^{2^{k}+\dots+2} \E^{+} (Q\G^k) |\G|^{-k/8} 
	=
	$$
	$$
	=
	(2^{3k} M^{2^k-2} C^{k/4}_* \log^{k}  |Q\G^{k-1}|)  
	\cdot |Q|^{2^{k+1}-2} \E^{+} (Q\G^k) |\G|^{-k/8} \,.
	$$
	Now 
	recalling 
	the assumption $|Q\G^{k+1}| |Q\G^k| |\G| \le p^2$ and applying Lemma \ref{l:AA_small_energy}, we get
	$$
	\E^{+}_{2^{k+1}} (Q) \le (M^{2^k+1} 2^{3k+1} C^{(k+4)/4}_* \log^{k}  |Q\G^{k-1}|)  \cdot
	|Q|^{2^{k+1}+1} |\G|^{-k/8-1/2} \,. 
	$$
	In particular, this final step covers the remaining case $l=1$ above. 
	This completes the proof. 
	$\hfill\Box$
\end{proof}


\begin{remark}
	Let $\G$ be a multiplicative subgroup and $Q\G=Q$. 
	Then by Theorem \ref{t:QG} if $|Q| |\G| \le p$ and a number $k_1$ is chosen as  $|\G|^{k_1/8+1/2} \ge |Q| \log^{k_1} |Q|$, then $\E^{+}_{2^{k_1+1}} (Q) \ll_{k_1} 	|Q|^{2^{k_1+1}}$. 
	Let us compare this with Theorem \ref{t:Q_shift}.
	By this result, choosing $k_2$ such that $|\G|^{\frac{k_2+2}{2}} \ge |Q| \log^{4k_2} |Q|$, we get 
	$\E^{+}_{2^{k_2+1}} (Q) \ll_{k_2} |Q|^{2^{k_2+1}} |\G|^{\frac{1}{2}}$. 
	After that applying the second part of Corollary \ref{c:Q_cap_Q+x} other $n:=2^{k_2+1}$ times, we obtain 
	$$
	\E^{+}_{2^{2k_2+2}} (Q) \ll_{k_2} |Q|^{2^{2k_2+2}} + \E^{+}_{2^{k_2+1}} (Q) (|Q||\G|^{-\frac{1}{4} \cdot 2^{-k_2}})^n
	\ll_{k_2}
	|Q|^{2^{2k_2+2}} + |Q|^{2^{k_2+1}} |\G|^{\frac{1}{2}} |Q|^n |\G|^{-\frac{1}{2}}
	\ll
	|Q|^{2^{2k_2+2}} \,.
	$$
	Thus, Theorem \ref{t:Q_shift} gives slightly  better  bound (in the case of multiplicative subgroups) but of the same form. 
\end{remark}


\begin{remark}
	From formula (\ref{tmp:12.04_1--}), it follows that for any $l$ one has  $\E^{+}_{l} (Q) \ge \frac{|Q|^{2l}}{p^{l-1}}$.
	Hence upper bound (\ref{f:QG'}) 
	can has place 
	just for small sets  $Q$. 
	For example, taking the smallest possible  $l=2$ and comparing $|Q|^2$ with $|Q|^4/p$ we see that the condition $|Q|<\sqrt{p}$ is enough. 
	If $Q = Q\G$, where $\G$ is a multiplicative subgroup, then it is possible to refine this condition because in the proof of Theorem \ref{t:Q_shift}
	another method (the Fourier approach) was used.
	We did not make such calculations. 
	\label{r:E_sqrt}
\end{remark}


Now we can obtain analogues of Corollaries \ref{c:Q_cap_Q+x} and  \ref{c:E(Q,A)}.

\begin{corollary}
	Let $\G, Q_1,Q_2 \subseteq \F_p$ be sets. 
	Take $k\ge 0$ such that for $j=1,2$ 
	one has\\ $|Q_j \G^{k+2}| |Q_j\G^{k+1}| |\G| \le p^2$, $|Q_j \G^{k}| |\G| \le p$,
	$|Q_j \G| \le M_* |Q_j|$, $|Q_j \G^{k+2}| \le M |Q_j|$, 
	and 
	\begin{equation}\label{cond:Q_cap_Q+x_M}
	|\G|^{k/8+1/2} \ge |Q_j| \cdot  M_* M^{2^k+1} 2^{3k+1} C^{(k+4)/4}_* \log^{k} |Q_j \G^k| \,.
	\end{equation}
	Then for any $x\neq 0$ the following holds
	\begin{equation}\label{f:Q_cap_Q+x_M}	
	|Q_1\cap (Q_2+x)| \le  2M_* M \sqrt{|Q_1||Q_2|} \cdot |\G|^{-\frac{1}{2} 2^{-k}} \,.
	\end{equation}
	\label{c:Q_cap_Q+x_M}
\end{corollary}
\begin{proof}
	Denote by $\rho$ the quantity $|Q_1\cap (Q_2+x)|$. 
	On the one hand, applying 
	the Cauchy--Schwarz inequality and 
	the second part of Theorem \ref{t:QG},  
	we obtain 
	$$
	\sum_{y} r^{2^{k+1}}_{\G Q_1 - \G Q_2} (y) 
	\le (\E^{+}_{2^{k+1}} (\G Q_1))^{1/2} (\E^{+}_{2^{k+1}} (\G Q_2))^{1/2} 
	\le
	$$
	$$
	\le
	2^{3k+2} M^{2^k+1} (|Q_1 \G| |Q_2 \G|)^{2^{k}} 
	\le
	2^{3k+2} M^{2^k+1} M^{2^{k+1}}_* (|Q_1| |Q_2|)^{2^{k}} \,.
	$$ 
	On the other hand,  it is easy to see that for any $y\in \G x$ one has
	$r_{\G Q_1 - \G Q_2} (y) \ge \rho$.
	Thus, 
	$$
	\rho^{2^{k+1}} |\G| \le 2^{3k+2} M^{2^k+1} M^{2^{k+1}}_* (|Q_1| |Q_2|)^{2^{k}} \,.
	$$ 
	and hence
	$$
	\rho \le  2M_* M \sqrt{|Q_1||Q_2|} \cdot |\G|^{-\frac{1}{2} 2^{-k}} \,.
	$$
	Here we have used inequality $k\ge 5$ which easily follows from $|\G| \le |Q_j \G|\le M|Q_j|$ and condition (\ref{cond:Q_cap_Q+x_M}). 
	This completes the proof. 
	$\hfill\Box$
\end{proof}

\bigskip

In the next two corollaries 
we show how 
to replace the condition  $|Q \G^k| \ll |Q|$ to a condition with a single multiplication, namely, $|Q\G| \ll |Q|$.

\begin{corollary}
	Let $\G, Q$ be subsets of  $\F_p$, $M\ge 1$ be a real number, $|Q \G| \le M|Q|$.
	Suppose that for $k\ge 1$ one has  
	$(2M)^{k+1} |Q| |\G| \le p$, and
	\begin{equation}\label{c:cond_E(Q,A)_M}
	|\G|^{k/8+1/2} \ge |Q| \cdot  (2M)^{(k+3)2^k} C^{(k+4)/4}_* \log^{k}  ((2M)^k |Q|) \,.
	\end{equation}
	Then for any $A\subseteq \F_p$ the following holds
	\begin{equation}\label{f:E(Q,A)_M}	
	|A+Q| \ge 2^{-3} |Q| \cdot \min \{ |A|, 2^{-(4+k)} M^{-(k+3)} |\G|^{\frac{1}{2} 2^{-k}}  \} \,,
	\end{equation}
	and for any $\a \neq 0$ 
	\begin{equation}\label{f:E(Q,A)'_M}	
	|A(Q+\a)| \ge 2^{-3}  |Q| \cdot \min \{ |A|, 2^{-(4+k)} M^{-(k+3)} |\G|^{\frac{1}{2} 2^{-k}} \} \,.
	\end{equation}	
	\label{c:E(Q,A)_M}
\end{corollary}
\begin{proof}
	Using Lemma \ref{l:Plunnecke_R}, find a set $X \subseteq Q$, $|X| \ge |Q|/2$ such that for any $l$ the following holds $|X\G^l| \le (2M)^l |X|$. 
	Also, notice that  $|X\G| \le |Q\G| \le 2M|X|$.
	We 
	apply Corollary \ref{c:Q_cap_Q+x_M} with $M=(2M)^{k+2}$, $M_* = 2M$ 
	and 
	see that for any $x\neq 0$ the following holds 
	$$
	|Q_1\cap (Q_2+x)| \le  2^{k+4} M^{k+3} \sqrt{|Q_1||Q_2|} \cdot |\G|^{-\frac{1}{2} 2^{-k}} 
	\,.
	$$
	Here $Q_1 =X$ and $Q_2=X$ or $Q_2 = \a X$. 
	Using the arguments from Corollary \ref{c:E(Q,A)}, we estimate the energies 
	$\E^{+} (A,X)$, $\E^{\times} (A,X+\a)$.
	In particular, we obtain lower bounds for the sumset from (\ref{f:E(Q,A)_M}) and the product set from (\ref{f:E(Q,A)'_M}). 
	It remains to check condition $(2M)^{2k+3} |Q|^2 |\G| \le p^2$. 
	But it follows from $(2M)^{k+1} |Q| |\G| \le p$ if $M\le |\G|/2$. 
	The last inequality 
	is a simple consequence of 
	(\ref{c:cond_E(Q,A)_M}). 
	This completes the proof. 
	$\hfill\Box$
\end{proof}



\bigskip

Now we prove an analogue of Corollary \ref{c:Q_cap_Q+x_M} where we require just $|Q_j \G|$, $j=1,2$ are small comparable to $|Q_j|$.   
For simplicity we formulate the next corollary in the situation $|Q'|=|Q|$ but of course general bound takes place as well.

\begin{corollary}
	Let $\G, Q, Q'$ be subsets of  $\F_p$, $|Q'|= |Q|$, $M\ge 1$ be a real number, $|Q\G|, |Q'\G| \le M|Q|$.
	Suppose that for $k\ge 1$ one has 
	$(2M)^{k+1} |Q| |\G| \le p$, and
	$$
	|\G|^{\frac{k}{8}+\frac{1}{2(k+4)}  } 
	\ge
	|Q| \cdot M^{(k+3)2^{k}} C^{(k+4)/4}_* \log^{k}  (|\G|^{\frac{k}{2(k+4)} 2^{-k}} |Q|) 
	$$
	Then for any $x\neq 0$ one has 
	\begin{equation}\label{f:QM_shift_more}	
	|Q\cap (Q'+x)| \le 	 4M|Q| \cdot |\G|^{-\frac{1}{2(k+4)} 2^{-k}} \,.
	\end{equation}	
	\label{c:QM_shift_more}
\end{corollary}
\begin{proof}
	Let $\t{Q} = Q\cap (Q'+x)$. 
	Then $|\t{Q} \G| \le |Q \G| \le M|Q| = M|Q|/|\t{Q}| \cdot |\t{Q}| := \t{M} |\t{Q}|$. 
	Similarly,  
	$|(\t{Q}-x) \G| \le |Q' \G| \le M|Q|$. 
	Applying the second part of Corollary \ref{c:E(Q,A)_M} with $\a=x$, $Q=\t{Q}$, $A=\G$ and $M=\t{M}$, we 
	get 
	$$
	M|Q| \ge |(\t{Q}-x) \G| 
	\ge
	2^{-(7+k)} |\t{Q}| \t{M}^{-(k+3)} |\G|^{\frac{1}{2} 2^{-k}} 
	=
	2^{-(7+k)} M^{-(k+3)} |\t{Q}|^{k+4} |Q|^{-(k+3)} |\G|^{\frac{1}{2} 2^{-k}} 
	$$
	provided
	$$
	|\G|^{k/8+1/2} 
	\ge
	|Q| \cdot  (2\t{M})^{(k+3)2^k} C^{(k+4)/4}_* \log^{k}  ((2\t{M})^k |Q|) 
	\ge
	$$
	$$ 
	\ge
	|\t{Q}| \cdot  (2\t{M})^{(k+3)2^k} C^{(k+4)/4}_* \log^{k}  ((2\t{M})^k |\t{Q}|) \,.
	$$
	It gives us
	\begin{equation}\label{f:t{Q}_tmp}
	|\t{Q}| \le 4M|Q| \cdot |\G|^{-\frac{1}{2(k+4)} 2^{-k}} \,.
	\end{equation}
	Now if inequality does not hold, then $\t{M}  \le |\G|^{\frac{1}{2(k+4)} 2^{-k}} /4$. 
	Hence the condition
	$$
	|\G|^{\frac{k}{8}+\frac{1}{2(k+4)}  } 
	\ge
	|Q| \cdot M^{(k+3)2^{k}} C^{(k+4)/4}_* \log^{k}  (|\G|^{\frac{k}{2(k+4)} 2^{-k}} |Q|) 
	$$
	is enough. 
	This completes the proof. 
	$\hfill\Box$
\end{proof}

\bigskip 

Now we are ready to prove the main asymmetric sum--product result of this section.

\begin{corollary}
	Let $A,B,C\subseteq \F_p$ be arbitrary sets, 
	and $k\ge 1$ be such that 
	$|A| |B|^{1+\frac{(k+1)}{2(k+4)}2^{-k}} \le p$ and 
	\begin{equation}\label{cond:ABC_my}
	|B|^{\frac{k}{8}+ \frac{1}{2(k+4)}} \ge |A| \cdot  C^{(k+4)/4}_* \log^{k}  (|A||B|) \,.
	\end{equation}
	Then
	\begin{equation}\label{f:ABC_my1}
	\max\{ |AB|, |A+C| \} \ge 2^{-3} |A| \cdot \min\{|C|, |B|^{\frac{1}{2(k+4)} 2^{-k}} \} \,,
	\end{equation}
	and for any $\a \neq 0$
	\begin{equation}\label{f:ABC_my2}
	\max\{ |AB|, |(A+\a)C| \} \ge 2^{-3} |A| \cdot \min\{|C|, |B|^{\frac{1}{2(k+4)} 2^{-k}} \} \,.
	\end{equation}
	Moreover, 
	\begin{equation}\label{f:ABC_my3}
	|AB| + \frac{|A|^2 |C|^2}{\E^{+} (A,C)} \ge 2^{-4} |A| \cdot \min\{|C|, |B|^{\frac{1}{4(k+4)} 2^{-k}} \} \,,
	\end{equation}
	and for any $\a \neq 0$ the following holds 
	\begin{equation}\label{f:ABC_my3'}
	|AB| + \frac{|A|^2 |C|^2}{\E^{\times} (A+\a,C)} \ge 2^{-4} |A| \cdot \min\{|C|, |B|^{\frac{1}{4(k+4)} 2^{-k}} \} \,,
	\end{equation}
	provided
	$$
	|B|^{\frac{k}{8} - \frac{1}{4} +  \frac{1}{4(k+4)}} \ge |A| \cdot   C^{(k+4)/4}_* \log^{k}  (|A||B|) \,.
	$$
	\label{c:ABC_my}
\end{corollary}
\begin{proof}
	We will prove just (\ref{f:ABC_my1}) because  the same arguments hold for (\ref{f:ABC_my2}). 
	Put $|AB| = M|A|$, $M\ge 1$. 
	Applying Corollary \ref{c:E(Q,A)_M} with $Q=A$, $\G = B$, $A=C$ and choosing $k$ such that 
	\begin{equation}\label{tmp:02.05_1}
	|B|^{k/8+1/2} \ge |A| \cdot  2^{(k+3)2^k} M^{(k+3)2^k} C^{(k+4)/4}_* \log^{k}  ((2M)^k |A|) \,,
	\end{equation}
	we obtain
	$$
	|A+C| \ge 2^{-3} |A| \cdot \min \{ |C|, 2^{-(k+4)} M^{-(k+3)} |B|^{\frac{1}{2} 2^{-k}} \} \,.
	$$
	Thus, by small calculations (which correspond to the optimal choice of the parameter $M$, namely, $M = M_0 = 2^{-2} |B|^{\frac{1}{2(k+4)} 2^{-k}}$), we get
	$$
	\max\{ |AB|, |A+C| \} \ge 2^{-3} |A| \cdot \min\{|C|, |B|^{\frac{1}{2(k+4)} 2^{-k}} \} \,.
	$$
	Substituting $M=M_0$ into (\ref{tmp:02.05_1}), we obtain our condition (\ref{cond:ABC_my}). 
	The condition  $(2M)^{k+1} |A| |B| \le p$ gives us $|A| |B|^{1+\frac{(k+1)}{2(k+4)}2^{-k}} \le p$.

	To prove (\ref{f:ABC_my3}),  (\ref{f:ABC_my3'}) we use Corollary \ref{c:QM_shift_more} instead of Corollary \ref{c:E(Q,A)_M} and apply the arguments of the proof of Corollary \ref{c:E(Q,A)}.
	We obtain
	$$
	\E^{+} (A,C), \E^{\times} (A+\a, C) \le 2|A||C| + 4 |A| |C|^2 \cdot M^{} \cdot |B|^{-\frac{1}{2(k+4)} 2^{-k}}
	$$
	After that it remains to choose $M$ optimally, $M=2^{-1}  |B|^{\frac{1}{4(k+4)} 2^{-k}}$.  
	This completes the proof. 
	$\hfill\Box$
\end{proof}


\bigskip

Notice that one cannot obtain any nontrivial bounds for 
$\min\{ \E^{\times} (A,B)$, $\E^{+} (A,C) \}$.
Just take $B$ equals a geometric progression, $C$ equals an arithmetic  progression, $|B| = |C|$ and $A = B\cup C$. 



\begin{remark}
	The results of this section take place in $\R$.
	In this case we do not need in any 
	conditions 
	containing 
	the characteristic $p$. 
\end{remark}

Corollary \ref{c:ABC_my} gives us a series of examples of "superquadratic expanders" \cite{BR-NZ} with four variables. 
The first example of such an expander with four variables was given in \cite{Rudnev_cross}.

\begin{corollary}
	Let $\_phi : \R \to \R$ be an injective function. 
	Then for any $\kappa < \frac{1}{40} 2^{-16}$ and an arbitrary finite set $A\subset \R$ one has $|R[A] \_phi (A)| \gg |A|^{2+\kappa}$. 
	In particular,
	$$
	R[A] A = \left\{ \frac{(y-x)w}{z-x} ~:~ x,y,z,w \in A,\, x \neq z  \right\}
	$$
	is a superquadratic expander with four variables. \\
	Moreover, for any finite sets $A,B,C,D$ of equal sizes one has 
	\begin{equation}\label{f:superquadratic_ABCD}
	\left|\left\{ \frac{(y-x)w}{z-x} ~:~ x\in A,\,y\in B,\,z\in C,\,w\in D,\, x \neq z  \right\} \right| \gg |A|^{2+\kappa} \,. 
	\end{equation}
	\label{c:superquadratic}
\end{corollary}
\begin{proof}
	By a result from \cite{J_PhD}, \cite{R_Minkovski}, we have $|R[A]| \gg |A|^2 / \log |A|$. 
	Further $R[A] = 1-R[A]$ and $R^{-1} [A] = R[A]$, see Remark \ref{r:2/3}. 
	Hence applying estimate (\ref{f:ABC_my2}) of Corollary \ref{c:ABC_my} with $A=R[A]$, $B=C=\_phi(A)$ and $\alpha = -1$,  we obtain 
	$$
	|R[A] \_phi (A)| \gg |R[A]| \cdot |A|^{\frac{1}{2(k+4)} 2^{-k}} \,,
	$$
	provided 
	\begin{equation}\label{tmp:12.05_1}
	|A|^{\frac{k}{8}+ \frac{1}{2(k+4)}} \ge |R[A]| \cdot  2^{2k} C^{(k+4)/4}_* \log^{k}  |A| 
	\ge
	\end{equation}
	$$
	\ge 
	|R[A]| \cdot  C^{(k+4)/4}_* \log^{k}  |R[A]\_phi (A)| \,.
	$$
	Put $|R[A]| = C |A|^{2+c} / \log |A|$, $c\ge 0$ and $C>0$ is an absolute constant.  
	Then taking $k=16+8c$, say, we satisfy (\ref{tmp:12.05_1}) for large $A$.  
	It follows that
	$$
	|R[A] \_phi (A)| \gg |A|^{2+c + \frac{1}{2(20+8c)} 2^{-16-8c}} \log^{-1} |A| \,.
	$$
	One can check that the optimal choice of $c$ is $c=0$. 
	%
	Finally, to prove (\ref{f:superquadratic_ABCD}) just notice that from 
	the method 
	of  \cite{J_PhD}, \cite{R_Minkovski}
	it follows that $\left| \left\{ \frac{b-a}{c-a} ~:~ a\in A,\, b\in B,\, c\in C,\, c \neq a \right\} \right| \gg |A|^2 / \log |A|$
	for any sets $A,B,C$ of equal cardinality. 
	After that repeat the arguments above. 
	This completes the proof. 
	$\hfill\Box$
\end{proof}




\begin{remark}
	Let us 
	show quickly how  
	Corollary \ref{c:ABC_my} implies both Theorems \ref{t:Bourgain_ABC}, \ref{t:Bourgain_ABC_shift} for sets $A$ with $|A| < p^{1/2-\eps}$
	(the appearance $\sqrt{p}$ bound was discussed in Remark \ref{r:E_sqrt}).

	Let $B,C$ be some sets of sizes 
	greater than 
	$p^{\eps}$
	such that $\max\{ |AB|, |A+C| \} \le p^\d |A|$ or $\max\{ |(A+\a)B|, |A+C| \} \le p^\d |A|$ for some $\a \neq 0$.
	We can find sufficiently large $k=k(\eps)$ such that  condition (\ref{cond:ABC_my}) takes place for $B$ because $|A| < p^{1/2-\eps} \le p$ and $|B|\ge p^\eps$.
	Applying Corollary \ref{c:ABC_my} for $A,B, C$, we arrive to a contradiction.
	Finally, to insure that  
	$|A| |B|^{1+\frac{(k+1)}{2(k+4)}2^{-k}} \le p$ 
	just use  the assumption 
	$|A| < p^{1/2-\eps}$, inequality $|B| \le |AB| \le p^{\d} |A|$ 
	and take sufficiently small $\d = \d(\eps)$ and  sufficiently large $k = k(\eps)$. 
	\label{r:implication}
\end{remark}


Let $A\subset \R$ be a finite set.
Consider a characteristic of $A$ (see, e.g., \cite{s_diff}), which generalize the notion of small multiplicative doubling of $A$.
%
%
%
%
%
%
Namely, put 
$$
d^{+} (A) := \inf_f \min_{B\neq \emptyset} \frac{|f(A)+B|^2}{|A| |B|} \,, 
$$
where the infimum is taken over convex/concave functions $f$.

{\bf Problem.} Suppose that $d^{+} (A) \le |A|^\eps$, $\eps>0$ is a small number. 
Is it true that there is $k=k(\eps)$ such that $\E^{+}_k (A) \ll |A|^{k}$?

Notice that one cannot obtain a similar bound for $\T^{+}_k (A)$.
Indeed, let $A = \{1^2, 2^2, \dots, n^2 \}$. 
Then one can show that for such $A$   the quantity $d^{+}(A)$ is $O(1)$ (see, e.g., \cite{s_diff}) 
but, clearly, $|kA| \ll_k |A|^2$. 
It means that 
it is not possible to obtain 
any upper bound for $\T^{+}_k (A)$ of the form $\T^{+}_k (A) \ll |A|^{2k-2-c}$, $c>0$ 
and hence any analogues of Theorems \ref{t:T_k_G_M}, \ref{t:T_k_G_MR} for sets $A$ with small  $d^{+} (A)$.

\bigskip

\noindent{I.D.~Shkredov\\
	Steklov Mathematical Institute,\\
	ul. Gubkina, 8, Moscow, Russia, 119991}
\\
and
\\
IITP RAS,  \\
Bolshoy Karetny per. 19, Moscow, Russia, 127994\\
and 
\\
MIPT, \\ 
Institutskii per. 9, Dolgoprudnii, Russia, 141701\\
{\tt ilya.shkredov@gmail.com}



\begin{thebibliography}{99}
	
	
	
	
	
	
	\bibitem{AMRS}
	{\sc E. Aksoy Yazici, B. Murphy, M. Rudnev, I. D. Shkredov, }
	{\em Growth Estimates in Positive Characteristic via Collisions, }
	Int. Math. Res. Not. IMRN, 2016. First published online, doi: 10.1093/imrn/rnw206.
	
	
	\bibitem{BR-NZ}
	{\sc  A. Balog, O. Roche-Newton, D. Zhelezov, }
	{\em Expanders with superquadratic growth, } arXiv:1611.05251v1 [math.CO] 16 Nov 2016.
	
	
	\bibitem{Bourgain_more}
	{\sc J.~Bourgain, }
	{\em More on the sum--product phenomenon in prime fields and its applications, }
	Int. J. Number Theory {\bf 1}:1 (2005), 1--32.
	
	\bibitem{Bourgain_01}
	{\sc J.~Bourgain, }
	{\em On the Erd\H{o}s-Volkmann and Katz-Tao ring conjecture, }
	GAFA 13 (2003), 334--365.
	
	
	\bibitem{Bourgain_DH}
	{\sc J.~Bourgain, }
	{\em Estimates on exponential sums related to the Diffie--Hellman distributions, } 
	GAFA {\bf 15}:1 (2005), 1--34.
	
	
	\bibitem{Bourgain_Z_q}
	{\sc J.~Bourgain, }
	{\em Exponential sum estimates over subgroups of $\Z^*_q$, $q$ arbitrary, } 
	J. Anal. Math. 97 (2005), 317--355.
	
	
	\bibitem{Bourgain_ideal}
	{\sc J.~Bourgain, }
	{\em Exponential sum estimates in finite commutative rings and applications, }
	J. Anal. Math. 101 (2007), 325--355.
	
	
	\bibitem{BC}
	{\sc J.~Bourgain, M.-C. Chang, } 
	{\em Exponential sum estimates over subgroups and almost subgroups of, where $Q$ is composite with few prime factors, } 
	GAFA {\bf 16}:2 (2006), 327--366.
	
	
	\bibitem{BG}
	{\sc J.~Bourgain, M.Z.~Garaev, }
	\emph{On a variant of sum-product estimates and explicit exponential sum bounds in prime fields, }
	Math. Proc. Cambridge Philos. Soc. 146 (2009), no.1, 1--21.
	
	
	
	\bibitem{BGK} 
	{\sc J.~Bourgain, A.~A.~Glibichuk, S.~V.~Konyagin, }
	{\em Estimate for the number of sums and products and for exponential sums in fields of prime order.}
	J. London Math. Soc. (2) 73 (2006), 380--398.
	
	
	\bibitem{BKT}
	{\sc J.~Bourgain, N.~Katz, T.~Tao,, }
	{\em A sum--product theorem in finite fields and applications, }
	GAFA,  {\bf 14}:1 (2004), 27--57.
	
	
	
	
	\bibitem{BC}
	{\sc A. Bush, E. Croot, }
	{\em Few products, many h--fold sums, }
	arXiv:1409.7349v4 [math.CO] 18 Oct 2014.
	
	
	
	\bibitem{ES}
	{\sc P.~Erd\H{o}s, E.~Szemer\'{e}di, }
	\emph{On sums and products of integers, }
	Studies in pure mathematics, 213--218, Birkh\"auser, Basel, 1983.
	
	
	\bibitem{Garaev_survey}
	{\sc M.Z.~Garaev, }
	\emph{Sums and products of sets and estimates of rational trigonometric sums in fields of prime order, }
	Uspekhi Mat. Nauk, {\bf 65}:4 (2010), 599--658.
	
	
	
	\bibitem{GK}
	{\sc A.A.~Glibichuk, S.V.~Konyagin, }
	{\em Additive Properties of Product Sets in Fields of Prime Order, }
	Additive combinatorics 43 (2007): 279--286.
	
	
	
	\bibitem{Harald}
	{\sc H.A.~Helfgott, } 
	{\em Growth and generation in $SL_2 (\Z/p\Z)$}, 
	Annals of Mathematics (2008), 601--623.
	
	
	
	\bibitem{J_PhD}
	{\sc T.G.F. Jones, }
	{\em New quantitative estimates on the incidence geometry and growth of finite sets, }
	PhD thesis, arXiv:1301.4853, 2013.
	
	
	\bibitem{K_mult}
	{\sc S.V.~Konyagin, }
	{\em h--fold Sums from a Set with Few Products, }
	MJCNT, {\bf 4}:3  (2014), 14--20.
	
	
	\bibitem{KS2}
	{\sc S.~V.~Konyagin, I.~D.~Shkredov, }
	\emph{New results on sum-products in $\R$, }
	Proc.  Steklov Inst. Math. {\bf 294} (2016), 87--98.
	
	
	
	\bibitem{KS1}
	{\sc S.~V.~Konyagin, I.~Shparlinski, }
	{\em Character sums with exponential functions, } Cambridge University Press, Cambridge, 1999.
	
	
	\bibitem{MPR-NRS}
	{\sc B. Murphy, G. Petridis, O. Roche-Newton, M. Rudnev, I. D. Shkredov, }
	{\em New results on sum-product type growth in positive characteristic, } arXiv:1702.01003v2  [math.CO]  8 Feb 2017.
	
	
	
	\bibitem{R_Minkovski}
	{\sc O.~Roche--Newton, }
	{\em  A short proof of a near--optimal cardinality estimate for the product of a sum set, }
	arXiv:1502.05560v1 [math.CO] 19 Feb 2015.
	
	
	\bibitem{RRS}
	{\sc O. Roche--Newton, M. Rudnev, I. D. Shkredov, }
	{\em New sum--product type estimates over finite fields, }
	Adv. Math. {\bf 293} (2016), 589--605.
	
	
	
	\bibitem{Rudnev_pp}
	{\sc M. Rudnev, }
	{\em On the number of incidences between planes and points in three dimensions, }
	Combinatorica, 2017. First published online, doi:10.1007/s00493-016-3329-6.
	
	
	
	\bibitem{Rudnev_cross}
	{\sc M. Rudnev, }
	{\em On distinct cross--ratios and related growth problems}
	arXiv:1705.01830v1 [math.CO] 4 May 2017.
	
	
	
	
	
	\bibitem{Ruzsa_book}
	{\sc I.Z. Ruzsa, }
	{\em Sumsets and structure, }
	Combinatorial number theory and additive group theory (2009): 87--210.
	
	
	
	\bibitem{SS1}  
	{\sc T. Schoen and I. D. Shkredov, }
	\emph{Higher moments of convolutions, }
	J. Number Theory 133 (2013), no. 5, 1693--1737.
	
	
	
	\bibitem{s_ineq} 
	{\sc I. D. Shkredov, }
	\emph{Some new inequalities in additive combinatorics, }
	MJCNT, {\bf 3}:2 (2013), 237--288.
	
	
	\bibitem{Shkr_H+L}  
	{\sc I. D. Shkredov, }
	{\em Energies and structure of additive sets, }
	Electronic Journal of Combinatorics, {\bf 21}:3 (2014), \#P3.44, 1--53.
	
	
	\bibitem{Shkr_tripling}  
	{\sc I. D. Shkredov, }
	{\em On tripling constant of multiplicative subgroups, }
	Integers, \#A75 {\bf 16} (2016), 1--10. 
	
	
	\bibitem{s_E_k}
	{\sc I. D. Shkredov, }
	{\em Some remarks on the Balog--Wooley decomposition theorem and quantities $D^+$, $D^\times$, }
	Proc. Steklov Math. Inst., accepted; arXiv:1605.00266v1 [math.CO] 1 May 2016. 
	
	
	\bibitem{s_diff}
	{\sc I. D. Shkredov, }
	{\em Difference sets are not multiplicatively closed, }
	Discrete Analysis, 17, (2016), 1--21. DOI: 10.19086/da.913.
	
	
	
	
	
	\bibitem{Steinikov_T_3}
	{\sc Yu. N. Shteinikov, }
	{\em Estimates of trigonometric sums over subgroups and some of their applications, } Mathematical Notes, {\bf 98}:3 (2015), 667--684.
	
	
	
	\bibitem{sz-t}
	{\sc E.~Szemer\' edi, W.~T.~Trotter,} {\em Extremal problems in
		discrete geometry,} Combinatorica 3 (1983), 381--392.
	
	
	
	\bibitem{TV}
	{\sc T.~Tao, V.~Vu, }{\em Additive combinatorics,} Cambridge University Press 2006.
	
	
	
\end{thebibliography}
\end{document}